\documentclass{article}
\usepackage{bbm}
\usepackage{geometry}
\usepackage{amsfonts}
\usepackage{latexsym, amssymb, amsmath, amscd, amsthm, amsxtra}
\usepackage{mathtools}
\usepackage{enumerate}
\usepackage[all]{xy}
\usepackage{mathrsfs}
\usepackage{fancyhdr}
\usepackage{listings}
\usepackage{hyperref}
\hypersetup{
     colorlinks=true
}

\newtheorem{thm}{Theorem}[section]
\newtheorem{prop}{Proposition}[section]

\newtheorem{Lem}[thm]{Lemma}
\newtheorem{lemma}[thm]{Lemma}
\newtheorem{cor}[thm]{Corollary}

\newcommand{\Ex}[1]{\left\langle {#1}\right\rangle}

\newcommand{\bn}{  \beta }
\def \P {\mathbb P}

\numberwithin{equation}{section}
\begin{document}

\title{Spectral Gap Estimates for Mixed $p$-Spin Models at High Temperature}

\author{Arka Adhikari\thanks{Department of Mathematics, Stanford University, Stanford CA 94305-2125, USA}
\and Christian Brennecke\thanks{Institute for Applied Mathematics, University of Bonn, 53115 Bonn, Germany}
\and Changji Xu\thanks{Center of Mathematical Sciences and Applications, Harvard University, Cambridge MA 02138,USA}
\and Horng-Tzer Yau\thanks{Department of Mathematics, Harvard University, Cambridge MA 02138, USA}}





\maketitle
\begin{abstract}
We consider general mixed $p$-spin mean field spin glass models and provide a method to prove that the spectral gap of the Dirichlet form associated with the Gibbs measure is of order one at sufficiently high temperature. Our proof is based on an iteration scheme relating the spectral gap of the $N$-spin system to that of suitably conditioned subsystems. 
\end{abstract}
\section{Introduction}
In this paper we consider general mixed $p$-spin mean field models with energies described by the Hamiltonian $H_N:\Sigma_N := \{-1,1\}^{N}\to\mathbb{R}$ of the form
\begin{equation*}
   H_N(\sigma):= \sum_{p\geq2} \frac{\beta_p}{N^{(p-1)/2}} \sum_{\substack{1 \leq i_1,\ldots,i_p \leq N }} g_{i_1 i_2\ldots i_p} \sigma_{i_1} \ldots \sigma_{i_p} + \sum_{1 \leq i \leq N} \eta_i \sigma_i,
\end{equation*}
for $\sigma = (\sigma_1,\sigma_2, \ldots,\sigma_N) \in\Sigma_N $. Here, the $g_{i_1 i_2\ldots i_p}$ are i.i.d. standard Gaussian random variables for all tuples $(i_1,i_2\ldots ,i_p)\in \mathbb{N}^p$ with $ i_{k}\neq  i_{l} $ for each $1\leq k\neq l\leq p$ and $p\in \mathbb{N}$. If two indices $ i_{k}=  i_{l}  $ coincide, we set $ g_{i_1\ldots i_k\ldots i_l\ldots i_p} \equiv 0$. Moreover, we assume that the temperature coefficients $\beta_p\geq 0$ are summable and the external field strengths are denoted by $\eta_i \in \mathbb{R}$. The Gibbs measure $\mu_N:\mathcal{P}(\Sigma_N)\to (0,1)$ that corresponds to $H_N$ is defined by
    $$\mu_{N}(\sigma): = \frac{1}{Z_N} e^{H_N(\sigma)}, \hspace{0.5cm}Z_N := \sum_{\sigma\in \Sigma_N} e^{H_N(\sigma)},$$ 
where $Z_N $ denotes the partition function of the system. In the following, we write $\langle \cdot \rangle$ and $\langle \cdot \,; \cdot \rangle$ for the expectation and covariance, respectively, related to $\mu_N$.

We are interested in the spectral gap of the Dirichlet form w.r.t. $\mu_N$. More precisely, denote the discrete partial derivative in direction $\sigma_i, i\in \{1,\dots, N\}$, by 
        $$(\partial_i f)(\sigma) := \frac{1}{2}\big(f(\sigma) - f(\hat{\sigma}_i)\big), \text{ where } \hat \sigma_i := (\sigma_1,...,-\sigma_i,...,\sigma_N).$$
Then we can write $H_N(\sigma) = B_j(\sigma) \sigma_j + H_N^{(j)}(\sigma)$ s.t. the $j$-th cavity field $ B_j:\Sigma_N\to \mathbb{R}$, defined by
        \begin{equation*}
             B_j(\sigma)=\sigma_j\partial_j H_N(\sigma),
        \end{equation*}
and $ H_N^{(j)} = H_N - B_j$ do not depend on $\sigma_j$ (notice that $\partial_i \sigma_j = \delta_{ij}\sigma_j $). Heuristically, one may think of $B_j$ as the effective magnetic field acting on $\sigma_j$ that is caused by the remaining spins $\sigma_i$ for $i\neq j$.
   
With these definitions, we define the Dirichlet form weighted by the cavity fields $B_j$ through
    \begin{equation}\label{eq:defDf}
    D(f) := \sum_{ j =1}^N \Ex{ \cosh^{-2}(B_j) (\partial_j f)^2}
    \end{equation}
for every $f:\Sigma_N\to \mathbb{R}$ and the (inverse of its) spectral gap by
\begin{equation*}
    a_{H_N}:= \sup_{\substack{f: \Sigma_N \to \mathbb R, \\ f\neq const. }} \frac{\langle f ; f\rangle}{ D(f)}.
\end{equation*}

Our main result is the following theorem.
\begin{thm}
\label{thm:main}
 For every $\epsilon > 0$, if $\bn:= \sum_{p\geq 2} \sqrt{p^3 \log p}\,\beta_p$ is sufficiently small (depending on $\epsilon$), then there exist constants $c , C >0$, independent of $N\in \mathbb{N}$, s.t. 
    $$ \mathbb{P} \big( \{a_{H_N} > 1 + \epsilon\} \big) \leq C  \exp(-c\, N). $$
\end{thm}

\vspace{0.2cm}
\noindent \textbf{Remarks:} 
\begin{enumerate}
    \item We clearly have $D(f)=0$ if and only if $f$ is constant, i.e. $f\equiv \langle f\rangle $. Projecting onto the orthogonal complement of such functions in $L^2(d\mu_N)$, observe that $a_{H_N}$ is indeed the inverse of the spectral gap of $D$ (extended in the obvious way to a non-negative quadratic form in $L^2(d\mu_N)$) above its zero ground state energy. Theorem \ref{thm:main} shows that, at sufficiently high temperature, the spectral gap is of order one. 
    \item Notice that, due to the weight factors $ \cosh^{-2}(B_j)\leq 1$, Theorem \ref{thm:main} implies also that the standard Dirichlet form in $L^2(d\mu_N)$ has a spectral gap of order one, at sufficiently high temperature. 
    \item Choosing linear functions $f = \sum_{j=1}^N c_j \sigma_j $ for $ \|c\|_2 \neq 0$, Theorem \ref{thm:main} implies in particular that the correlation matrix $M = (\langle \sigma_i;\sigma_j\rangle)_{1\leq i,j\leq N}$ has operator norm bounded by $1+ \epsilon$ with high probability, if $\bn $ is small enough.  
\end{enumerate}

Spectral gap inequalities appear most prominently in the study of Markov chains. The difference between the largest and second largest eigenvalues of the Markov chain essentially dictates how fast the chain can equilibriate; with a large spectral gap, one can expect exponentially fast equilibriation. Markov chain Monte Carlo is one of the most efficient tools used to model physical systems. Thus, establishing a spectral gap inequality for a particular Markov chain model of interest would be critical to show numerical modelers that their results would result in accurate predictions of the behavior of these statistical models here. One classical Markov chain algorithm is the Glauber dynamics; the spectral gap we prove shows that the Glauber dynamics on the hypercube for the $p$-spin spin glass will equilibriate very quickly.
 
 Due to its importance in the numerical sampling of statistical models, spectral gap inequalities have been of central interest to researchers. In the special case $\beta_2 > 0 $ and $\beta_p=0$ for $p>2$, our model reduces to the well-known Sherrington-Kirkpatrick model \cite{SK}. In this case, a logarithmic Sobolev inequality (LSI) follows for $\beta < 1/4$ from the main result of \cite{BB19}, whose proof is based on a single-step renormalisation and Bakry-\'Emery theory \cite{BaEm}.  However, the proof applied in this case is not readily generalizable to general mixed $p$-spin models. Indeed, the quadratic nature of the Sherrington-Kirkpatrick model allowed the authors of \cite{BB19} to `complete the square' on the Sherrington-Kirkpatrick Hamiltonian (in fact, the results of \cite{BB19} apply to general quadratic models) and represent it as the convolution of a continuous model with a coupled Bernoulli variable. At this point, one can prove an LSI separately for each part by classical methods and then prove an LSI for the convolution through a short computation. The fact that the interaction in $H_N$, defined above, is not quadratic means that this method fails at the first step and there seems to be no natural way to decompose the $p$-spin model into simpler models as in \cite{BB19}.

The paper \cite{EKZ22,AJKPV21,CE22},  prove spectral gap inequalities or LSI via a method of stochastic localization.
These methods allow the authors to interpolate the Sherrington-Kirkpatrick model to a mixture of 2-spin Ising models with a rank 1 quadratic interaction.
This reduces the question of the spectral gap of the SK model to the  spectral gap of a simpler model. Since the $p$-spin models have a more complicated
spin interaction, it is unclear that current methods would allow one to easily decompose the interaction of the $p$-spin model to simpler parts. 
The paper \cite{AMS22} combines stochastic localization with approximate message passing in order to sample distributions from the Sherrington-Kirkpatrick models; this method is different from Markov chain Monte Carlo. The analysis in this work used a comparison between the Sherrington-Kirkpatrick Hamiltonian and the planted model; this used the quadratic interaction of the SK Hamiltonian, and an analysis of this form does not easily extend to the mixed $p$-spin model. 
 For spectral gap estimates in the physically very different low temperature regime, see, however, \cite{ArJa} and the references therein, where it is shown under rather general conditions that a spectral gap inequality as implied by Theorem \ref{thm:main} does not hold true. 

Since these methods do not readily extend to $p$-spin case, we use a different approach inspired by ideas related to self-consistent relations and martingale arguments as introduced in \cite{LY93,LY98}.
Theorem \ref{thm:main} is the first spectral gap bound for general mixed $p$-spin models at sufficiently high temperature.

\section{Outline of the Proof of Theorem \ref{thm:main}}
A simple computation shows that the spectral gap of a system with only one spin is equal to one (see Section \ref{sec:pfmain} below for the short argument). In order to estimate the spectral gap of the $N$-spin system, we proceed iteratively over the system size. To this end, let us first introduce some additional notation. For disjoint subsets $ A, B\subset \{1,\dots, N\}$, we define $ H_N^{ [A,B]} \equiv H_{N,\sigma_A}^{[A,B]}:  \Sigma_{N -|A\cup B| }\to\mathbb{R}$ by
        \[H_{N, \sigma_A }^{[A,B]}( \sigma_{(A \cup B)^c}) =  \sum_{p\geq2} \frac{\beta_p}{N^{(p-1)/2}} \sum_{\substack{ i_1,\ldots,i_p \in B^c}} g_{i_1 i_2\ldots i_p} \sigma_{i_1} \ldots \sigma_{i_p} + \sum_{ i \in B^c} \eta_i \sigma_i,  \]
where $S^c = \{1,\dots,N\}\setminus S $ for $ S\subset \{1,\dots, N\}$, $\sigma_S$ denotes $\sigma_S := (\sigma_i)_{i \in S}$ (in particular, we identify $ \sigma_i = \sigma_{\{i\}}$) and where the coordinates of $\sigma_A$ are understood to be fixed. Observe that $ H_N^{[A,B]}$ plays the role of the energy of the subsystem consisting of the spins $\sigma_i, i \in (A\cup B)^c$, conditionally on the spins $ \sigma_j $ for $j\in A$ and with the particles $\sigma_j$ for $j\in B$ removed from the system. We then denote by $ \langle \cdot \rangle_{[A,B]} \equiv \langle \cdot \rangle_{[A,B]}(\sigma_A)$ the conditional Gibbs measure induced by the reduced Hamiltonian $ H_N^{[A,B]}$ and, in analogy to \eqref{eq:defDf}, we set
    \begin{equation*}
    D_{[A,B]}(f) := \sum_{ j  \not \in A \cup B}\Ex{ \cosh^{-2}(B_j^{[A,B]}) (\partial_j f)^2}_{[A,B]},
    \end{equation*}
where $B_j^{[A,B]} := \sigma_j \partial_j H_N^{[A,B]} $. Notice that $B_j^{[A,B]} $ is formally obtained by setting $\sigma_B$ equal to zero and fixing $\sigma_A$ in $B_j$. More explicitly, it is given by
        \begin{equation}
            \label{eq:BjAB}
            \begin{split}
                B_j^{[A,B]}(\sigma) = &\;\sum_{p\geq2} \frac{\beta_p}{N^{(p-1)/2}}\bigg( \sum_{\substack{ i_2,\ldots,i_p \in B^c }} g_{j i_2\ldots i_p} \sigma_{i_2} \ldots \sigma_{i_p} +\ldots  \\
                &\hspace{2.5cm}\ldots+ \!\!\!\! \sum_{\substack{ i_1,\ldots,i_{p-1} \in B^c}} g_{i_1 i_2\ldots i_{p-1}j} \sigma_{i_1}\sigma_{i_2} \ldots \sigma_{i_{p-1}}  \bigg) +   \eta_j  
            \end{split}
        \end{equation}
for $\sigma = (\sigma_A, \sigma_{A^c\cap B^c})\in \Sigma_{N-|B|}$, with $\sigma_A \in \Sigma_{|A|}$ understood as being fixed. 

Next, we define the spectral gap of the subsystem related to $A$, $B$ by
\begin{equation}
\label{eq:def-aAB}
    a_{[A,B]} \equiv a_{[A,B]}(\sigma_A) :=  \sup_{\substack{ f:\Sigma_{N -|A\cup B| }\to \mathbb{R}, \\f\neq const.}} \frac{\langle f ; f\rangle_{[A,B]}}{D_{[A,B]}(f)}\,,
\end{equation}
and the maximal spectral gap over all $N-k$ spin subsystems by
\begin{equation}\label{eq:defamax}
a_{N-k}: =  \max_{\substack{A,B\subset \{1,\dots, N\}:\\A\cap B=\emptyset, |A\cup B|=k }} \max_{\sigma_A \in \Sigma_{|A|}}a_{[A,B]}(\sigma_A)\,.
\end{equation}
For $  i, j  \in (A\cup B)^c$, we finally set $m^{[A,B]}_i = \Ex{\sigma_i}_{[A,B]}$ and $m^{[A,B]}_{ij} = \Ex{\sigma_i;\sigma_j}_{[A,B]}$. 

Let us point out that, throughout the rest of this work, we keep the dependence of quantities like $ \langle \cdot\rangle^{[A,B]}, H_N^{[A,B]}$, etc. on $\sigma_A \in \Sigma_{|A|}$ implicit, to ease the notation. 

Our starting point for the proof of Theorem \ref{thm:main} is the following lemma.
\begin{Lem}[Conditioning Lemma]\label{lem:Condition}
Let $f:\Sigma_{N-k}\to \mathbb{R}$ and let $A,B $ be disjoint with $|A \cup B | = k$, then 
\begin{equation}
\label{eq:condition}
    \langle f; f \rangle_{[A,B]} \leq \left(1-\frac{1}{N-k}\right)a_{N - k-1} D_{[A,B]}(f) + \frac{1}{N-k} \sum_{j \not \in A \cup B} \frac{\langle f; \sigma_j \rangle^2_{[A,B]}}{1- (m^{[A,B]}_j)^2}.
\end{equation}
\end{Lem}
This inequality links the spectral gap $a_{[A,B]}$ over the $N-k$ spin system related to $A$ and $B$ with the maximal spectral gap $a_{N-k-1}$ over the $N-k-1$ spin systems, and our proof of Theorem \ref{thm:main} is based on employing this relation inductively. The key point is then to show that the second term on the r.h.s. in \eqref{eq:condition} is under good control during the iteration. Heuristically, note that if we write $ f=\sum_{j\not \in A\cup B} c_j \sigma_j + h $ such that $ \langle h;\sigma_j\rangle = 0 $ (that is, up to a constant, $h$ is the projection of $f$ onto the orthogonal complement of the functions $\sigma\mapsto \sigma_j - m_j^{[A,B]}$ in $L^2(d\langle \cdot\rangle_{[A,B]})$), then 
        $$ \sum_{j \not \in A \cup B} \frac{\langle f; \sigma_j \rangle^2_{[A,B]}}{1- (m^{[A,B]}_j)^2} = \| (\Lambda^{[A,B]})^{1/2} M^{[A,B]} \mathbf c\|_2^2  $$
for $\Lambda^{[A,B]} = \mathrm{diag}\big((1 - (m^{[A,B]}_i)^2)^{-1} \big)$ and $M^{[A,B]} = (m_{ij}^{[A,B]})_{i,j \not \in A \cup B}$. In particular, motivated by well-known results on the correlation matrix like for example \cite{Hanen}, if we could ignore the order one entries on the diagonal of the correlation matrix $M^{[A,B]} $, we might expect the right hand to be of size $\epsilon\langle f; f\rangle$ for small $\epsilon$ if the inverse temperature coefficients are small enough, and with such a bound one could easily iterate \eqref{eq:condition} to conclude Theorem \ref{thm:main}. 

Of course, we can not simply ignore the diagonal of $M^{[A,B]} $ and therefore, we need to proceed slightly differently. First, in Section \ref{sec:3}, we give the rough bound 
		$$\sum_{j\not \in A\cup B} \frac{\langle f; \sigma_j \rangle^2_{[A,B]}}{1- (m^{[A,B]}_j)^2} \leq C \sqrt{N-k}\langle f; f \rangle_{[A,B]} $$
on a set of probability close to one, uniformly in the subsets $A,B$. This bound implies that the maximal spectral gap can not have a large jump after adding one additional spin into the system. In the second step, we then control the error term in \eqref{eq:condition}  through an improved iteration bound which, loosely speaking, has the form
 		$$ \sum_{j\not \in A\cup B}\frac{\langle f; \sigma_j \rangle^2_{[A,B]}}{1- (m^{[A,B]}_j)^2} \leq  \big(1+O(\epsilon)\big) D_{[A,B]}(f) + O(\epsilon)\langle f; f \rangle_{[A,B]}$$
 for small $\epsilon$ if the temperature coefficients are small and if we have a some a priori control on $ a_{N-k-1}$. To obtain this bound, it turns out that we can follow a similar heuristics as above for the correlation matrix, but with the correlation matrix replaced by another matrix whose norm indeed turns out to be small at high temperature. Equipped with this improved estimate and the continuity argument from the first step, we use \eqref{eq:condition} inductively and conclude Theorem \ref{thm:main} in Section \ref{sec:pfmain}. 
 
 Let us now conclude this overview with the proof of \eqref{eq:condition}.
\begin{proof}[\bf Proof of Lemma \ref{lem:Condition}]
By the total variance formula, we have for any $j \not \in A \cup B$
    \begin{equation}
    \begin{split}
    \label{eq:totalvar}
    \langle f; f \rangle_{[A,B]} &= \langle \langle f(\cdot, \sigma_j);f(\cdot, \sigma_j) \rangle_{[A \cup \{j\},B]} \rangle_{[A,B]} \\
    & \hspace{.5cm}+  \langle \langle f(\cdot, \sigma_j) \rangle_{[A \cup \{j\},B]} ; \langle f(\cdot, \sigma_j) \rangle_{[A \cup \{j\},B]} \rangle_{[A,B]}\,.
    \end{split}
    \end{equation}
By definition \eqref{eq:defamax}, we can bound the integrand in the first term on the r.h.s. in \eqref{eq:totalvar} by 

        \[\begin{split}
         &\langle \langle f(\cdot, \sigma_j);f(\cdot, \sigma_j) \rangle_{[A \cup \{j\},B]} \rangle_{[A,B]} \\
         &\le a_{N - k-1} \sum_{l \not \in A\cup\{j\} \cup B}  \Ex{ \langle\cosh^{-2}(B_l^{[A\cup\{j\},B]}) (\partial_l f(\cdot, \sigma_j))^2 \rangle_{[A \cup \{j\},B]} }_{[A,B]}  \\
        & = a_{N - k-1} \sum_{l \not \in A\cup\{j\} \cup B}  \Ex{\cosh^{-2}(B_l^{[A ,B]}) (\partial_l f )^2 }_{[A,B]}\,, 
        \end{split}\]
so that averaging over $j\not \in A\cup B $ implies
\begin{align*}
    \langle f; f \rangle_{[A,B]} \leq &\left(1-\frac{1}{N-k}\right)a_{N - k-1} D_{[A,B]}(f) \\
    &+ \frac{1}{N-k} \sum_{j \not \in A \cup B} \langle \langle f(\cdot, \sigma_j) \rangle_{[A\cup\{j\},B]}; \langle f(\cdot, \sigma_j) \rangle_{[A\cup\{j\},B]} \rangle_{[A,B]}.
\end{align*}
Finally, a straightforward computation shows that
        $$\Ex{g(\sigma_j);g(\sigma_j)}_{[A,B]} = \frac{\Ex{g(\sigma_j);\sigma_j}_{[A,B]}^2}{1 - (m^{[A,B]}_j)^2} $$
for every function $\Sigma_{N-k}\ni \sigma\mapsto g (\sigma)=g(\sigma_j)$ that only depends on $\sigma_j$. Choosing $\sigma\mapsto g (\sigma)= \langle f(\cdot, \sigma_j) \rangle_{[A\cup\{j\},B]} $, we obtain
    \begin{equation*}
    \begin{split}
    \langle   g; g  \rangle_{[A,B]} &= \frac{\langle \langle f(\cdot, \sigma_j) \rangle_{[A\cup\{j\},B]}; \sigma_j \rangle^2_{[A,B]}}{1- (m^{[A,B]}_j)^2}=  \frac{\langle f; \sigma_j \rangle^2_{[A,B]}}{1- (m^{[A,B]}_j)^2}\,.
    \end{split}
    \end{equation*}
\end{proof}

\section{Continuity argument}\label{sec:3}  

As mentioned in the previous section, our first goal is to show that, assuming $a_{N-k-1}$ to be of order one, the rate with which $a_{N-k}$ may increase is not too large, with high probability. To make this more precise, let us set from now on
		\begin{equation}\label{eq:defbeta}
		\beta := \sum_{p\geq 2} \sqrt{p^3 \log p}\,\beta_p 
		\end{equation}
and let us define the good event $\Omega$ by
    \begin{equation}
    \label{eq:def-omg}
    \Omega: = \bigg\{  \sup_{ \substack{  A\cap B=\emptyset  }}\sup_{\sigma_A\in \Sigma_{|A|} } \sup_{\sigma \in \Sigma_{N-|A\cup B|} } \big\| \big(\partial_i B_j^{[A,B]} (\sigma)\big)_{ 1 \leq i,j \leq N-|A\cup B|}\big\| \leq  5 \bn\bigg\}\,.
  \end{equation}
Then, we prove in Section \ref{sec:est} below that $ \mathbb{P}(\Omega)\geq 1-e^{-N}$ (see Lemma \ref{lem:partialbnd}). In the following, let us also denote by $ C_\beta$ the constant
        \begin{equation} 
        \label{eq:def-cbeta}
        C_\beta: = (10^3 \bn )^2 \exp(10^3 \bn),  
        \end{equation}
whose specific form becomes clearer in Section \ref{sec:est}. The main result of this section reads as follows. 
\begin{prop}
\label{continuity}
Let $\bn $ be as in \eqref{eq:defbeta}, $\Omega$ as in \eqref{eq:def-omg}, $C_\beta$ as in \eqref{eq:def-cbeta} and let $\epsilon \in (0,10^{-2)}$. Assume that $\beta$ is sufficiently small. Then, there exists a universal constant $C>0$ such that
        \[\begin{split}&  \mathbb{P}\bigg( \Omega \cap \bigcup_{k=0}^{N-2}\Big\{ a_{N-k-1} C_\beta <\epsilon\text{ and }a_{N-k} >  5a_{N-k-1}  \Big\}  \bigg) \\
        & \hspace{0.5cm}\leq   C\,e^{( C \epsilon/C_\beta +2) \log N + N \log 4  -  N   C_\beta^{2}(\epsilon^{2} C\beta^2)^{-1}  }. \end{split} \]
\end{prop}
Prop. \ref{continuity} requires a couple of auxiliary results and follows by combining Lemma \ref{indu-M} and Corollary \ref{ct-M} below. We start with the following observation.
\begin{lemma}For any disjoint $A, B\subset \{1,\dots, N\}$ with $|A \cup B| = k$, we have that
\label{indu-M}
\begin{equation*}
    \sum_{i\not \in A\cup B} \frac{\Ex{\sigma_i;f}_{[A,B]}^2}{1 - (m^{[A,B]}_i)^2} \leq || (\Lambda^{[A,B]})^{1/2} M^{[A,B]}(\Lambda^{[A,B]})^{1/2}||\langle f; f \rangle_{[A,B]}\,,
\end{equation*}
where $\Lambda^{[A,B]} = \big(\delta_{ij}\big[ 1 - (m_i^{[A,B]})^2\big]^{-1}\big)_{i,j \not \in A \cup B}$ and $M^{[A,B]} = (m_{ij}^{[A,B]})_{i,j \not \in A \cup B}$.
As a result, $a_{[A,B]}$ is bounded by
  \begin{equation}
\label{eq:indu-M}
  \Big(1 -  \frac{|| (\Lambda^{[A,B]})^{1/2} M^{[A,B]}(\Lambda^{[A,B]})^{1/2}||}{N - k}\Big)a_{[A,B]}\leq a_{N-k-1}\,.
\end{equation}
\end{lemma}

\begin{proof}
We write $f = \sum_{i  \not \in A \cup B} c_i \sigma_i  + h $ with $  \langle h;\sigma_i \rangle_{[A,B]} = 0$ for all $i  \not \in A \cup B$. Notice that, up to a constant, $h$ is the projection of $f$ onto the orthogonal complement of the functions $\sigma\mapsto \sigma_i - m_i^{[A,B]} $ in $L^2(d \langle\cdot\rangle_{[A,B]})$. Then

\begin{equation}
\label{eq:f=csigma+h}
        \begin{aligned}
    \sum_{i\not \in A\cup B} \frac{\Ex{\sigma_i;f}_{[A,B]}^2}{1 - (m^{[A,B]}_i)^2}  & =\sum_{i\not \in A\cup B}\frac{( \sum_j m^{[A,B]}_{ij}c_j )^2}{1 -( m^{[A,B]}_i)^2} =  \| (\Lambda^{[A,B]})^{1/2} M^{[A,B]} \mathbf c\|_2^2  \\
    &\leq ||(M^{[A,B]})^{1/2} \Lambda^{[A,B]} (M^{[A,B]})^{1/2} ||  \langle \mathbf c, M^{[A,B]} \mathbf c \rangle_2 \\
    & \leq ||(M^{[A,B]})^{1/2} \Lambda^{[A,B]} (M^{[A,B]})^{1/2}  ||\langle f; f \rangle_{[A,B]}, \end{aligned}
\end{equation}
where from now on $ \langle\cdot, \cdot\rangle_2 $ denotes the standard Euclidean inner product, $\|\cdot\|_2$ is its induced norm and $\|\cdot\|$ denotes the matrix norm, i.e. the maximal eigenvalue in case of a symmetric, positive semi-definite matrix like $(M^{[A,B]})^{1/2} \Lambda^{[A,B]} (M^{[A,B]})^{1/2} $. Noting that $||(M^{[A,B]})^{1/2} \Lambda^{[A,B]} (M^{[A,B]})^{1/2}  || = || (\Lambda^{[A,B]})^{1/2} M^{[A,B]}(\Lambda^{[A,B]})^{1/2}||$, we conclude the claim using Lemma \ref{lem:Condition} and the definition of $a_{[A,B]}$ in \eqref{eq:def-aAB}.
\end{proof} 

To proceed further, we need some a priori information on the distribution of the cavity fields $ B_j^{[A,B]}$. In essence, the next result allows us to control the Gibbs expectation of exponentials of $ B_j^{[A,B]}$ by the exponentials evaluated at the Gibbs expectation of $ B_j^{[A,B]}$. To focus on the main line of the argument for the proof of Theorem \ref{thm:main}, we defer the proof of the following technical key lemma to Section \ref{sec:est}. 

\begin{lemma}\label{bd-m}
Let $\beta$ be as in \eqref{eq:defbeta}, $\Omega$ as in \eqref{eq:def-omg}, $\epsilon \in (0,10^{-2})$, $A, B\subset \{1,\dots, N\}$ disjoint with $|A \cup B| = k$ and let $j\not \in A\cup B$. Assume that $\beta$ is sufficiently small. If $ a_{[A,B\cup \{j\}]}  C_\beta<\epsilon$, then we have in $\Omega$ that
        \begin{align}
        \label{eq:bd-m-1}
        &\frac{1}{ \big\langle \cosh^2(B^{[A,B]}_j)\big\rangle_{[A,B\cup\{j\}]}} \leq  (1+4\epsilon) \big( 1 - (m^{[A,B]}_j)^2\big),\\
             \label{eq:bd-m-3}
         & \big\langle e^{-K B^{[A,B]}_j}\big\rangle_{[A,B\cup\{j\}  ]} \leq (1 + 4\epsilon) e^{-K\big\langle B^{[A,B]}_j\big\rangle_{[A,B\cup\{j\} ]}} , \,\\
               \label{eq:bd-m-2}
        &1 \leq  \frac{\big\langle\cosh(KB^{[A,B]}_j)\big\rangle_{[A,B\cup\{j\}]}}{\cosh\big(K\big\langle{B^{[A,B]}_j}\big\rangle_{[A,B\cup\{j\}]}\big)} \leq 1 + 4\epsilon
               \end{align}
uniformly in $ K\in [-20, 20]$.
\end{lemma}
 
We use Lemma \ref{bd-m} to prove the following result.  
  
\begin{lemma}
\label{ct-sg} 
Let $\Omega$ be as in \eqref{eq:def-omg}, $C_\beta$ as in \eqref{eq:def-cbeta}, $\epsilon \in (0,10^{-2})$ and set $$ \Omega_{N-k-1,\epsilon} = \{ a_{N-k-1}  C_\beta < \epsilon\}.$$ Then, there exists a universal constant $C>0$ such that in the event $\Omega\cap \Omega_{N-k-1,\epsilon}$, we have for every disjoint $A,B$ with $|A\cup B| = k$, $j \not \in A \cup B$, every $\sigma_A\in \Sigma_{|A|}$ and every function $g:\Sigma_{N-k} \to \mathbb{R} $ that is independent of $\sigma_j$ that 
        \[ \Ex{\sigma_j;g}_{[A,B]}^2 \leq  \frac{C\epsilon}{C_\beta} \big(1 - (m^{[A,B]}_j)^2\big)^2 D_{[A,B\cup\{j\}}(g)\,.\]
  
\end{lemma}

\begin{proof}
Assume in the following that $g:\Sigma_{N-k} \to \mathbb{R} $ is a function that is independent of $\sigma_j$. Then, writing $H_{N}^{[A,B]}= H_{N}^{[A,B \cup\{j\}]} + \sigma_j B_j^{[A,B]}$, we have that
        $$\begin{aligned}
        \Ex{\sigma_j;g}_{[A,B]} = &\frac{\Ex{ \sum_{\sigma_j\in \{\pm1\}} \sigma_j g e^{\sigma_j B_j^{[A,B]}}}_{[A,B\cup\{j\}]} }{ \Ex{2\cosh(B_j^{[A,B]})}_{[A,B\cup\{j\}]}}     \\&- \frac{\Ex{\sum_{\sigma_j\in \{\pm 1\}}\sigma_j  e^{\sigma_j B_j^{[A,B]}}}_{[A,B\cup\{j\}]}}{ \Ex{2\cosh(B_j^{[A,B]})}_{[A,B\cup\{j\}]} }\frac{   \Ex{\sum_{\sigma_j\in \{\pm1\}}g e^{\sigma_j B_j^{[A,B]}}}_{[A,B\cup\{j\}]}}{\Ex{2\cosh(B_j^{[A,B]})}_{[A,B\cup\{j\}]}}\,.
        \end{aligned}$$
To simplify the notation, let us abbreviate for the rest of the proof
  \[\Ex{\cdot}_*=\Ex{\cdot}_{[A,B\cup\{j\}]},\quad\xi = \frac{\sinh(B_j^{[A,B]})}{\big\langle\cosh(B_j^{[A,B]})\big\rangle_{[A,B\cup\{j\}]}},\quad \zeta  = \frac{\cosh(B_j^{[A,B]})}{\big\langle\cosh(B_j^{[A,B]})\big\rangle_{[A,B\cup\{j\}]}},\]
so that, by the independence of $g$ of $\sigma_j$, we find
 \begin{align*} 
   \Ex{\sigma_j;g}_{[A,B]} 
    = \Ex{\xi g}_{*} - \Ex{\xi}_{*}\Ex{g \zeta}_{*}
    &=\Ex{\xi;g}_{*} - \Ex{\xi}_{*}\Ex{g;\zeta}_{*}=\big\langle\big(\xi - \Ex{\xi}_{*}\zeta\big) ;g\big\rangle_{*}.
 \end{align*}
 Notice that in the second step, we used $ \langle\zeta\rangle_* =1$. Observing also that
        \[ \Ex{\xi}_{*} = \frac{\big\langle\sinh(B_j^{[A,B]})\big\rangle_{[A,B\cup\{j\}]}}{\big\langle\cosh(B_j^{[A,B]})\big\rangle_{[A,B\cup\{j\}]}} = \langle \sigma_j\rangle_{[A,B]} = m_j^{[A, B]}, \]
 we apply Cauchy-Schwarz and obtain the upper bound
    \begin{equation}
    \label{eq:aux1}
    \begin{split}
    \Ex{\sigma_j;g}_{[A,B]}^2 &\leq \Ex{g;g}_{*} \Ex{\xi - m^{[A,B]}_j \zeta;\xi - m^{[A,B]}_j \zeta}_{*} \\
    &\leq a_{N-k-1} D_{[A,B\cup\{j\}]}(g) \Ex{(\xi - m^{[A,B]}_j \zeta)^2}_{*}\,.
    \end{split}
    \end{equation}
In $\Omega_{N-k-1}$, we know that $a_{N-k-1} C_\beta<\epsilon$, so the claim follows if we show that the last factor on the r.h.s. of the previous estimate is bounded by some constant $C>0$ that is independent of all parameters. To this end, we bound
        \begin{equation*}
        \begin{split}
        &\Ex{(\xi - m^{[A,B]}_j \zeta)^2}_{*}\\
        &= \Ex{\cosh(B_j^{[A,B]})}_{*}^{-2}\Ex{\cosh^2(B_j^{[A,B]})\big(\tanh(B_j^{[A,B]}) - m^{[A,B]}_j\big)^2}_{*}\\
        & \leq \Ex{\cosh(B_j^{[A,B]})}_{*}^{-2} \Ex{\cosh^4(B_j^{[A,B]})}_{*}^{1/2}   \Ex{\big(\tanh(B_j^{[A,B]}) - m^{[A,B]}_j\big)^4}_{*}^{1/2}\,.
        \end{split}
        \end{equation*}
In $\Omega\cap\Omega_{N-k-1,\epsilon}$ we can apply Lemma \ref{bd-m}, because $a_{N-k-1} C_\beta<\epsilon$ implies in particular that $ a_{[A,B\cup \{j\}]} C_\beta<\epsilon$. Expanding $\cosh(x) = (e^x+e^{-x})/2 $, this yields 
    \begin{equation}\label{eq:aux3}\Ex{\cosh(B_j^{[A,B]})}_{*}^{-2} \Ex{\cosh^4(B_j^{[A,B]})}_{*}^{1/2} \leq C\end{equation}
and we also claim that
\begin{equation}
\label{eq:uabd}
    \Ex{\big(\tanh(B_j^{[A,B]}) - m^{[A,B]}_j\big)^4}_{*} \leq C \big(1 -(m_j^{[A,B]})^2\big)^4\,.
\end{equation}
Assuming the validity of \eqref{eq:uabd} for the moment, the lemma follows by combining the previous four estimates. Hence, let us focus on the proof of \eqref{eq:uabd}. As explained below, this bound follows by giving a uniform (in $N$) lower bound on $1 + m^{[A,B]}_j$ or $1 - m^{[A,B]}_j$, depending on whether $\langle B_j^{[A,B]}\rangle_*\geq 0$ or $\langle B_j^{[A,B]}\rangle_*\leq 0$. To see this, let us assume first that $\langle B_j^{[A,B]}\rangle_* \ge 0$. Then, Lemma \ref{bd-m} implies
 
    \begin{equation*} \label{eq:probabilitybnd}
    \begin{split}
    \big\langle\mathbf 1\big\{B_j^{[A,B]}\leq - 1/2\big\}\big\rangle_* &\le \exp (- K/2 ) \big\langle\exp\big(- K    B_j^{[A,B]}\big)\big\rangle_*\\ &\le (1+\epsilon) \exp (- K/2 )\exp\big(- K   \langle B_j^{[A,B]}\rangle_*\big) \le 2\exp(-K/2).
    \end{split}
    \end{equation*}
for every $0\leq K\leq 20$. Choosing $0\leq K\leq 20$ suitably, we get
    \begin{equation*} \label{eq:mxlwrbnd}
    \begin{split}
    1 + m^{[A,B]}_j  &= 1 + \big\langle\mathrm{tanh}\big(B_j^{[A,B]}\big)\big\rangle_{[A,B]}\\
    &\ge 1 + \mathrm{tanh}\big(-1/2\big) -  \frac{ \langle \sinh^2(B_j^{[A,B]})\rangle_*^{1/2}}{\langle \cosh(B_j^{[A,B]})\rangle_*} \big\langle\mathbf 1\big\{B_j^{[A,B]}\leq - 1/2\big\}\big\rangle_*^{1/2}\\
    &\geq 1 - \mathrm{tanh}\big(1/2\big) -2\exp(-K/4)\geq 1/(1+e) >0.
    \end{split}
    \end{equation*}
Using this lower bound, the elementary bounds $ 0\leq 1-\tanh(x)\leq 2e^{-2x}$ and $\cosh^2(x)\leq \cosh(2x)$ for all $x\in\mathbb{R}$ and once again Lemma \ref{bd-m}, we find that
        \begin{align*}
        \big\langle\big(\tanh(B_j^{[A,B]}) - m^{[A,B]}_j\big)^4\big\rangle_{*}  \leq & \;C \big\langle\big(1-\tanh(B_j^{[A,B]}) \big)^4\big\rangle_{*} + C\big (1 - m_j^{[A,B]}\big)^4\\
        \leq &\; C \big\langle e^{-8 B_j^{[A,B]}} \big\rangle_* + C \big(1 -(m_j^{[A,B]})^2\big)^4\\ 
        \leq &\; C e^{-8  \langle B_j^{[A,B]} \rangle_*} + C \big(1 -(m_j^{[A,B]})^2\big)^4
        \end{align*}
as well as
\begin{equation*}
   e^{-2 \Ex{B_j}_*} \leq \frac12 \cosh\big(2\langle B_j^{[A,B]}\rangle_*\big)^{-1} \leq C \big\langle\cosh\big(2B_j^{[A,B]}\big)\big\rangle_{[A,B\cup \{j\}]}^{-1} \leq C\big(1 -(m_j^{[A,B]})^2\big),
\end{equation*}
which, combined with the previous bound, implies \eqref{eq:uabd} for $\langle B_j^{[A,B]}\rangle_* \ge 0$. 

If $\langle B_j^{[A,B]}\rangle_* < 0$, we proceed similarly as above and the bound \eqref{eq:uabd} follows from 
 estimating $ \langle\mathbf 1\{B_j^{[A,B]}\geq 1/2\}\rangle_* \leq 2 \exp(-K/2)$ so that $ 1 - m^{[A,B]}_j \geq 1/(1+e)$ by choosing a suitable $0\leq K\leq 20$, and combining this with
        \begin{align*}
        \big\langle\big(\tanh(B_j^{[A,B]}) - m^{[A,B]}_j\big)^4\big\rangle_{*}  \leq & \;C \big\langle\big(1+\tanh(B_j^{[A,B]}) \big)^4\big\rangle_{*} +  C \big(1 + m_j^{[A,B]}\big)^4\\
        \leq &\; C e^{8  \langle B_j^{[A,B]} \rangle_*} + C \big(1 -(m_j^{[A,B]})^2\big)^4\\
        \leq &\; C \cosh\big(2\langle B_j^{[A,B]}\rangle_*\big)^{-4} + C \big(1 -(m_j^{[A,B]})^2\big)^4\\
        \leq &\;C \big(1 -(m_j^{[A,B]})^2\big)^4.
        \end{align*}
\end{proof}
\begin{cor}\label{ct-M}
Let $\Omega$ be as in \eqref{eq:def-omg}, $C_\beta$ as in \eqref{eq:def-cbeta}, $\epsilon\in (0,10^{-2})$ and set 
        $$ \Omega_{N-k-1,\epsilon} = \{ a_{N-k-1}  C_\beta < \epsilon\}.$$ 
Then, there exists a universal constant $C>0$, s.t. in $\Omega\cap \Omega_{N-k-1,\epsilon}$, we have for every $A,B$ be disjoint with $|A\cup B| = k<N$ and every $\sigma_A\in \Sigma_{|A|} $ that
            \[ \| (\Lambda^{[A,B]})^{1/2} M^{[A,B]}(\Lambda^{[A,B]})^{1/2}\| \leq  C\sqrt{ \frac{ \epsilon }{ C_\beta }} \sqrt{N - k}.\]
        Here, we recall that $\Lambda^{[A,B]}=\Lambda^{[A,B]}(\sigma_A) = \big(\delta_{ij}\big[ 1 - (m_i^{[A,B]})^2\big]^{-1}\big)_{i,j \not \in A \cup B}$ as well as $M^{[A,B]}= M^{[A,B]}(\sigma_A)= (m_{ij}^{[A,B]})_{i,j \not \in A \cup B}$. 
\end{cor}
\begin{proof} 
For $\mathbf c = (c_i)_{i \not \in A \cup B} \in \mathbb R^{N - k}$, we have that
    \begin{equation*}
    \begin{split}
    &\mathbf c^{\mathsf T}  (\Lambda^{[A,B]})^{1/2} M^{[A,B]}(\Lambda^{[A,B]})^{1/2} \mathbf c \\
    &= \sum_{j\not \in A\cup B} \frac{c_j}{\sqrt{1 - (m^{[A,B]}_j)^2}} \Ex{ \sigma_j; \sum_{i \not \in A\cup B} \frac{c_i \sigma_i }{\sqrt{1 - (m^{[A,B]}_i)^2}} }_{[A,B]}\,.
    \end{split}
    \end{equation*}
Fixing $j\not \in A\cup B$ and setting $ g_j(\sigma) := \sum_{i \not \in A\cup B\cup\{j\}} \big(1 - (m^{[A,B]}_i)^2\big)^{-1/2} c_i \sigma_i $, it is clear that $g_j$ is independent of $\sigma_j$ and Lemma \ref{ct-sg} implies that 
    \[ \begin{split}
    0&\leq\mathbf c^{\mathsf T}  (\Lambda^{[A,B]})^{1/2} M^{[A,B]}(\Lambda^{[A,B]})^{1/2} \mathbf c\\
    & \leq C\sqrt{ \frac{ \epsilon }{ C_\beta }} \sum_{j\not \in A\cup B}\sqrt{1 - (m^{[A,B]}_j)^2} |c_j|     \sqrt{\sum_{i \not \in A\cup B\cup\{j\}}\frac{ \big\langle\cosh^{-2}\big(B_i^{[A,B]}\big)\big\rangle_{[A,B]} c_i^2}{1 - (m^{[A,B]}_i)^2}}  +\|\textbf{c}\|_2^2
    \end{split}\]
Noting also that 
        \[\begin{split}
        \big\langle\cosh^{-2}\big(B_i^{[A,B]}\big)\big\rangle_{[A,B]} &= 1 - \big\langle\tanh^{2}\big(B_i^{[A,B]}\big)\big\rangle_{[A,B]} \\
        & \leq 1 - \big\langle \tanh\big(B_i^{[A,B]}\big)\big\rangle_{[A,B]}^2 = 1 - (m^{[A,B]}_i)^2, 
        \end{split}\]
we obtain 
        \begin{align*}
        \mathbf c^{\mathsf T}  (\Lambda^{[A,B]})^{1/2} M^{[A,B]}(\Lambda^{[A,B]})^{1/2} \mathbf c  \leq C\sqrt{ \frac{ \epsilon }{ C_\beta }} \sqrt{N-k} \;\|\textbf{c}\|_2^2.
        \end{align*}
Since the constant $C>0$ on the right hand side in the last bound does neither depend on $ A, B$ nor on the spin configuration $ \sigma_A$ (implicitly contained in the expectations w.r.t. the conditional Gibbs measure $\langle\cdot\rangle_{[A,B]} $), this proves the lemma. 
\end{proof}

The previous corollary is not, yet, enough to conclude Prop. \ref{continuity}. To this end, we need another consequence of Lemma \ref{bd-m}.

\begin{lemma}
\label{ct-sg2}
Let $\beta $ be as in \eqref{eq:defbeta}, $\Omega$ as in \eqref{eq:def-omg}, $C_\beta$ as in \eqref{eq:def-cbeta} and let $\epsilon\in (0,10^{-2})$. Assume that $\beta$ is sufficiently small and set 
        $$ \Omega_{N-k-1,\epsilon} = \{ a_{N-k-1}  C_\beta < \epsilon\}.$$
Moreover, let $ T\in\mathbb{N} $ be fixed (independently of $N$), let $ N- T\leq k \leq N-1  $ and define for $A,B\subset \{1,\dots,N\}$ disjoint with $ |A\cup B| =k$ and $\sigma_A\in \Sigma_{|A|}$ the event          \[\Omega_{A,B,\sigma_A}:=\Big\{ \big\| (\Lambda^{[A,B]})^{1/2}        M^{[A,B]}(\Lambda^{[A,B]})^{1/2}\big\|  > \frac{11}{10} \sqrt{N -k}   \Big\}. \]
Then, there is a universal constant $C >0$ such that
        \[\begin{split}  &\mathbb{P}\bigg(\bigcup_{\substack{ A\cap B=\emptyset,\\ |A\cup B|=k }} \bigcup_{\sigma_A \in \Sigma_{|A|}} \big(\Omega\cap \Omega_{N-k-1,\epsilon} \cap\Omega_{A,B,\sigma_A} \big)\bigg)  \\
        &\hspace{0.5cm} \leq C\,e^{( T+2) \log N + N \log 4  -  N   C_\beta(  CT\epsilon\beta^2)^{-1}  }.\end{split} \]
\end{lemma}
\begin{proof}
Let $T\in\mathbb{N}$ be fixed and $A,B\subset \{1,\dots, N\} $ be disjoint with $N-T\leq k=|A\cup B| \leq N - 1$. Now, consider the auxiliary field $  B_j^{[A,A^c]}$ which, formally, is obtained from $ B_j^{[A,  B]}$  by setting $ \sigma_{A^c }=0$ (note that $B\subset A^c$). Then, $  B_j^{[A,A^c]}$ is obviously a function of $\sigma_A$ alone and by \eqref{eq:BjAB}, we have that
        \[\begin{split}
         &(B_j^{[A,B]} -  B_j^{[A,A^c ]})(\sigma) \\
         & = \sum_{p\geq2} \frac{\beta_p}{N^{(p-1)/2}}\bigg( \sum_{\substack{ i_2,\ldots,i_p \in B^c,\\ \exists \; 2\leq l\leq p: \;i_l \in A^c\cap B^c}} g_{j i_2\ldots i_p} \sigma_{i_2} \ldots \sigma_{i_p} +\ldots  \\
                &\hspace{3cm}\ldots+ \!\!\!\! \sum_{\substack{ i_1,\ldots,i_{p-1} \in B^c,\\ \exists \; 1\leq l\leq p-1: \;i_l \in A^c\cap B^c}} g_{i_1 i_2\ldots i_{p-1}j} \sigma_{i_1}\sigma_{i_2} \ldots \sigma_{i_{p-1}}  \bigg) 
        \end{split}\]
for every $\sigma = (\sigma_A, \sigma_{A^c\cap B^c})\in \Sigma_{N-|B|} $. 

Moreover, given $\sigma\in \Sigma_{N-|B|} $, it is straightforward to see that $ \big(B_j^{[A,B]} -  B_j^{[A,A^c]}\big) (\sigma) $ is a centered Gaussian random variable with variance bounded by
        \[ \mathbb{E}\big(B_j^{[A,B]} - B_j^{[A,A^c]}\big)^2 (\sigma)\leq C\sum_{p\geq2 } \frac{N^{p-2}}{N^{p-1}} p\beta_p^2 T   \leq \frac{CT\beta^2 }{N} \]
for some universal constant $C>0$. Setting  
        \[\Omega_{\delta}:= \bigg\{ \sup_{\substack{ A,B\subset\{1,\dots, N\}:\\ A\cap B=\emptyset , |A\cup B| \geq N -T, }} \sup_{j \not \in A \cup B}\sup_{ \sigma \in \Sigma_{N-|B|}} \big|\big(B_j^{[A,B]} -    B_j^{[A,A^c]}\big)(\sigma)\big|  > \delta\bigg\} \]
for $\delta \in (0,1)$, we conclude that
        \begin{equation}\label{eq:expdec1}\begin{split} 
        \mathbb{P} (\Omega_{\delta}) &\leq  \binom{N}{T}T^2 (N-T)2^{N-T}2^{N-|B|} e^{-N \delta^2/(2CT\beta^2)}\\
        &\leq \exp \Big( ( T+2) \log N + N \log 4  -\frac{ N \delta^2}{CT\beta^2}  \Big).
        \end{split}\end{equation}
Now consider the event $ \Omega\cap \Omega_{N-k-1,\epsilon} \cap \Omega_\delta^c$. If $A,B $ are  disjoint s.t. $ |A\cup B| =k$, $\sigma_A\in \Sigma_{|A|}$ and $  g:\Sigma_{N-k} \to \mathbb{R} $ is a function independent of $\sigma_j$, for some $j\not \in A\cup B$, we can proceed as in \eqref{eq:aux1} to \eqref{eq:aux3} of Lemma \ref{ct-sg} to conclude that 
        \[ \Ex{\sigma_j;g}_{[A,B]}^2 \leq C a_{N-k-1} D_{[A,B\cup\{j\}]}(g) \Ex{\big(\tanh(B_j^{[A,B]}) - m^{[A,B]}_j\big)^4}_{[A,B\cup\{j\}]}^{1/2} \]
for some universal $C>0$ . By definition of $\Omega_\delta$, we have 
         \[\begin{split}
        &C \Big\langle\big(\tanh\big(B_j^{[A,B]}\big) - m^{[A,B]}_j\big)^4\Big\rangle_{*} \\
        & =    C \Big\langle\Big(\tanh\big(B_j^{[A,B]}\big) -  \tanh\big(B_j^{[A,A^c]}\big) \\
        &\hspace{1.2cm}+   \big\langle \tanh\big(B_j^{[A,A^c]}\big) - \tanh\big(B_j^{[A,B]}\big)\big\rangle_{[A,B]}\Big)^4\Big\rangle_{*} \leq C \delta^4
        \end{split}\]
in $\Omega\cap \Omega_{N-k-1,\epsilon}\cap\Omega_{\delta}^c$, where we used that $ B_j^{[A,A^c]}$ is constant, conditionally on $\sigma_A$, and that $x\mapsto \tanh(x)$ in $\mathbb{R}$ is globally Lipschitz (with Lipschitz constant bounded by one). That is, in $\Omega\cap \Omega_{N-k-1,\epsilon}\cap\Omega_{\delta}^c$, we have that 
        \[ \begin{split}
            \Ex{\sigma_j;g}_{[A,B]}^2 &\leq \frac{C\epsilon}{C_\beta} \delta^2 \big(1 - (m^{[A,B]}_j)^2\big)^2  D_{[A,B\cup\{j\}]}(g)
        \end{split} \]
uniformly in $A,B, \sigma_A$ and $g$. Therefore, arguing as in the proof of Corollary \ref{ct-M}, we find that in $\Omega\cap \Omega_{N-k-1,\epsilon}\cap\Omega_{\delta}^c$, we have
        \[ \sup_{\|\textbf{c}\|_2=1 }\mathbf c^{\mathsf T}  (\Lambda^{[A,B]})^{1/2} M^{[A,B]}(\Lambda^{[A,B]})^{1/2} \mathbf c  \leq C\delta\sqrt{\frac{\epsilon}{C_\beta}}  \sqrt{N-k}\leq \frac{11}{10} \sqrt{N-k}   \]
for all $\delta>0$ small enough so that $ \delta \leq C^{-1}\big(\frac{C_\beta}{\epsilon}\big)^{1/2}  $. Since this bound is true uniformly in the sets $A,B$ with $|A\cup B|=k $ (with $N-T\leq k \leq N-1)$ and $ \sigma_A\in\Sigma_{|A|}$, we conclude that 
        \[ \bigcup_{\substack{ A\cap B=\emptyset,\\ |A\cup B|=k }} \bigcup_{\sigma_A \in \Sigma_{|A|}} \big(\Omega\cap \Omega_{N-k-1,\epsilon} \cap\Omega_{A,B,\sigma_A} \big) \subset \Omega\cap \Omega_{N-k-1,\epsilon}\cap \Omega_{\delta_0} \subset \Omega_{\delta_0} \]
for  $ \delta_0 = C^{-1}\big(\frac{C_\beta}{\epsilon}\big)^{1/2}$. Hence, the bound \eqref{eq:expdec1} (for $ \delta = \delta_0$) concludes the lemma.
\end{proof}

We can now combine the previous lemmas to prove the main result of this section. 
\begin{proof}[\bf Proof of Prop. \ref{continuity}]
By Eq. \eqref{eq:indu-M} from Lemma \ref{indu-M} and by Corollary \ref{ct-M}, there exists a universal constant $C>0$ s.t. for all $A,B$ disjoint with $|A\cup B| = k<N$, we have in $\Omega\cap \Omega_{N-k-1,\epsilon} $ that
        \[ a_{[A,B]}=a_{[A,B]}(\sigma_A) \leq \Big(1- C\sqrt \epsilon/\sqrt{C_\beta(N-k)}\Big)^{-1} a_{N-k-1}. \]
Here, we set $\Omega_{N-k-1,\epsilon} = \{ a_{N-k-1}  C_\beta < \epsilon\}$ as before. Taking the sup over the spin configurations $ \sigma_A\in \Sigma_{|A|}$ and over all subsystems described by disjoint sets $A,B\subset \{1,\dots, N\}$ with $|A\cup B| = k<N$, this implies that in $\Omega\cap \Omega_{N-k-1,\epsilon} $
        \[ a_{N-k} \leq 5 a_{N-k-1} \]
for all $ k< N - T $, where $T :=  (5C/4)^2 (\epsilon/C_\beta)$. In other words, for such $k$, we have  
        \[ \Omega\cap \Omega_{N-k-1,\epsilon} \cap \{ a_{N-k}> 5a_{N-k-1} \} = \emptyset. \]
For the remaining $  N - T \leq k \leq N-2$, on the other hand, Lemma \ref{indu-M} and the fact that $ \frac{11}{10\sqrt 2}\leq \frac{4}5 $ imply that
        \[\begin{split} &\Omega\cap \Omega_{N-k-1,\epsilon} \cap   \big\{ a_{N-k}> 5a_{N-k-1} \big\}  \subset  \bigcup_{\substack{ A\cap B=\emptyset,\\ |A\cup B|=k }} \bigcup_{\sigma_A \in \Sigma_{|A|}} \big(\Omega\cap \Omega_{N-k-1,\epsilon} \cap\Omega_{A,B,\sigma_A} \big),
        \end{split}\]
where  
        \[\Omega_{A,B,\sigma_A}:=\Big\{ \big\| (\Lambda^{[A,B]})^{1/2} M^{[A,B]}(\Lambda^{[A,B]})^{1/2}\big\|(\sigma_A) > \frac{11}{10} \sqrt{N - |A \cup B|}   \Big\}. \]
By the preceding Lemma \ref{ct-sg2}, the probability of this event is bounded by
        \[\begin{split}
        \mathbb{P}\bigg(\bigcup_{\substack{ A\cap B=\emptyset,\\ |A\cup B|=k }} \bigcup_{\sigma_A \in \Sigma_{|A|}} \big(\Omega\cap \Omega_{N-k-1,\epsilon} \cap\Omega_{A,B,\sigma_A} \big) \bigg)\bigg) &\leq C\,e^{( T+2) \log N + N \log 4  -  N   C_\beta (  CT\epsilon\beta^2)^{-1}  }\\
        &\leq  C\,e^{( C \epsilon/C_\beta +2) \log N + N \log 4  -  N   C_\beta^{2}( C\epsilon^{2}\beta^2)^{-1}  }
         \end{split}\]
for a universal constant $C>0$. Finally, taking a union bound over these remaining $k$ with $  N - T \leq k \leq N-2$, we conclude the claim.  
\end{proof}

\section{Improved Iteration Estimate}
As pointed out in the introduction, the main difficulty in deriving an upper bound on the spectral gap $a_{H_N}$ lies in the fact that a priori it is not simple how to control the error term on the r.h.s. in \eqref{eq:condition}. While in the previous section, we have used the two point correlation matrix to control the error term, this is not, yet, enough to iterate \eqref{eq:condition} and obtain a meaningful upper bound on $a_{H_N}$. However, controlling the error term both through the Dirichlet form and the norm of a related matrix, we get sufficiently strong control on the error that allows us to iterate \eqref{eq:condition} and to give, in combination with Prop. \ref{continuity}, an inductive proof of Theorem \ref{thm:main}. The main result of this section reads as follows. 
\begin{prop} \label{thm:Dichotomy}
Let $\beta $ be as in \eqref{eq:defbeta}, $\Omega$ as in \eqref{eq:def-omg} and $C_\beta$ as in \eqref{eq:def-cbeta}. For $ \epsilon \in (0,10^{-2})$, set 
		\[\Omega_{N-k-1, \epsilon} = \big\{ a_{N-k-1}  \,C_\beta < \epsilon \big\}.\]
Then, for $\beta$ sufficiently small, we have in the event $\Omega\cap \Omega_{N-k-1,\epsilon}$ that
        \begin{equation*}
        \begin{split}
       &  \bigg(1 -  \frac{ C\beta^2 e^{C\beta} \, \max(1, a_{N-k})^2 }{ \epsilon  \,(N-k)} \bigg)  a_{N-k} \leq \Big(1- \frac{1}{N-k}\Big)a_{N-k-1} + \frac{(1+4\epsilon )^5}{N-k} 
        \end{split}
        \end{equation*}
for some universal constant $C>0$.      
\end{prop}
 

To derive the above result, our starting point is once again the upper bound \eqref{eq:condition}. However, in this section we relate the error term (the second term on the r.h.s. in \eqref{eq:condition}) to the Dirichlet form, up to another error that is indeed small at sufficiently high temperature. To be more precise, let $f:\Sigma_{N-|A\cup B|}\to \mathbb{R}$, then 
        \begin{align*}
        &\Ex{f;\sigma_j}_{[A,B]} -\Ex{\partial_j f;\sigma_j}_{[A,B]} \\
        & = \frac{1}{2} \Big\langle\big(f(\sigma) + f(\hat \sigma_j)\big) \sum_{\sigma_j\in \{\pm1\}} \frac{ (\sigma_j - m^{[A,B]}_j) e^{\sigma_jB_j^{[A,B]}}}{2\big\langle\cosh\big(B_j^{[A,B]}\big)\big\rangle_{[A,B\cup\{j\}]}}\Big\rangle_{[A,B\cup\{j\}]}   \\
        & = \frac{1}{2}\Big\langle  f(\sigma)  \Big( (\sigma_j - m^{[A,B]}_j)   -   e^{-2\sigma_j B_j^{[A,B]}} ( \sigma_j+m^{[A,B]}_j )\Big) \Big\rangle_{[A,B]} =: \frac{1}{2}\big\langle f h_j^{[A,B]}\big\rangle_{[A,B]}\,,
        \end{align*}
where we set 
        \begin{equation*}
        h_j^{[A,B]} = ( \sigma_j - m^{[A,B]}_j ) - e^{-2\sigma_j B_j^{[A,B]}} ( \sigma_j+m^{[A,B]}_j ). 
        \end{equation*}
By choosing $ f \equiv 1$, we observe that $ \big\langle h_j^{[A,B]}\big\rangle_{[A,B]}=0$, so that in fact
        \begin{equation}\label{eq:hjid}\Ex{f;\sigma_j}_{[A,B]} = \Ex{\partial_j f;\sigma_j}_{[A,B]} + \frac{1}{2}\big\langle f; h_j^{[A,B]}\big\rangle_{[A,B]} \end{equation}
for general $f:\Sigma_{N-|A\cup B|}\to \mathbb{R}$. The last identity shows that the error term in \eqref{eq:condition} can be controlled by $D_{[A,B]}(f)$, through the first term on the r.h.s. in \eqref{eq:hjid}, and, as explained below, by the norm of $S^{[A,B]} = (S^{[A,B]}_{ij})_{1\leq i,j\leq N-|A \cup B|}$, defined by
        \begin{equation} \label{eq:def-H}
        S^{[A,B]}_{ij} = \frac{\langle h^{[A,B]}_i; h^{[A,B]}_j \rangle_{[A,B]}}{ \sqrt{1-(m^{[A,B]}_i)^2} \sqrt{1- (m^{[A,B]}_j)^2}}.
        \end{equation}
The crucial observation is that the operator norm of $S^{[A,B]}$ is small if $\beta$ is sufficiently small, assuming some rough a priori information on $a_{N-k-1}$, like in the previous section. We make this more precise in the following auxiliary lemmas and conclude Prop. \ref{thm:Dichotomy} at the end of this section by combining Lemma \ref{lm:41} and Lemma \ref{lem:MatrixBound} below.   

\begin{lemma}\label{lm:41} Let $\epsilon\in (0,10^{-2})$, $\Omega$ be as in \eqref{eq:def-omg}, $C_\beta$ as in \eqref{eq:def-cbeta} and set 
        $$ \Omega_{N-k-1,\epsilon} = \{ a_{N-k-1}  C_\beta < \epsilon\}.$$
Assume that $\beta$ is sufficiently small. Then, for disjoint $A, B\subset \{1,\dots, N\}$ s.t. $|A\cup B| = k$, we have in $ \Omega\cap \Omega_{N-k-1,\epsilon}$ that
\begin{equation*}
    \sum_{j\not \in A\cup B} \frac{\Ex{\sigma_j;f}_{[A,B]}^2}{1 - (m^{[A,B]}_j)^2} \leq  (1 + 4\epsilon )^5 D_{[A,B]}(f)     +  \frac{1}{2 \epsilon }  \|S^{[A,B]}\|  \langle f; f \rangle_{[A,B]}\,.
\end{equation*}
As a consequence, we have in $ \Omega\cap \Omega_{N-k-1,\epsilon}$ that
    \begin{equation*} 
    \Big(1 -  \frac{1}{2 \epsilon }\frac{\| S^{[A,B]}\|}{N-k }\Big) a_{[A,B]} \le  \Big( 1- \frac{1}{N-k} \Big)  a_{N-k-1} + \frac{ (1 + 4\epsilon)^5 }{N-k}  \,.
\end{equation*}
 
\end{lemma}
\begin{proof}
Let $j\not \in A\cup B$. From \eqref{eq:hjid}, we obtain that 
	\begin{equation*}
    \langle f ;\sigma_j \rangle_{[A,B]}^2 \le (1 + \epsilon )  \langle \partial_j f; \sigma_j \rangle_{[A,B]}^2+ \Big(\frac{1}{4} + \frac{1}{4 \epsilon } \Big) \langle f; h_j^{[A,B]} \rangle^2_{[A,B]}.
	\end{equation*}
Also, since $\sigma_j\partial_j f$ does not depend on $\sigma_j$, we get
        \begin{align*}
        \big\langle\partial_j f;\sigma_j\big\rangle_{[A,B]}  &= \big\langle(1 - \sigma_j m^{[A,B]}_j)\sigma_j\partial_j f\big\rangle_{[A,B]}  \\
        & = \big\langle\big(1 - m^{[A,B]}_j\tanh\big(B_j^{[A,B]}\big) \big)\sigma_j\partial_j f\big\rangle_{[A,B]},
        \end{align*}
so that by Cauchy-Schwarz
        \[\begin{split}
            \big\langle\partial_j f;\sigma_j\big\rangle_{[A,B]}^2 & \leq \big\langle \cosh^{-2}\big(B_j^{[A,B]} \big) (\partial_j f)^2 \big\rangle_{[A,B]} \\
            &\hspace{0.5cm}\times\big\langle \big( \cosh\big(B_j^{[A,B]} \big) - m^{[A,B]}_j\sinh\big(B_j^{[A,B]} \big)^2 \big\rangle_{[A,B ]}. 
        \end{split}\]
Now, splitting $\cosh(x) - y\sinh(x) = \frac12 (1-y)e^x + \frac12 (1+y)e^{-x} $ for $x,y\in \mathbb{R}$ and recalling that $m^{[A,B]}_j = \langle \sinh (B_j^{[A,B]}  )    \rangle_{[A,B\cup\{j\}]}/\langle \cosh (B_j^{[A,B]}  )    \rangle_{[A,B\cup\{j\}]} $, we bound 
        \begin{equation} \label{eq:lm41aux2}
        \begin{split}
        & \frac14 \frac{ \big\langle \cosh \big(B_j^{[A,B]}  \big) - \sinh \big(B_j^{[A,B]}  \big)   \big\rangle_{[A,B\cup \{j\}]}^2 }{\big\langle \cosh \big(B_j^{[A,B]}  \big)    \big\rangle_{[A,B\cup\{j\}]}^2 }   \frac{ \big\langle  \cosh \big(B_j^{[A,B]}  \big)\exp \big(2 B_j^{[A,B]} \big)\big\rangle_{[A,B\cup\{j\}]} }{\big\langle \cosh \big(B_j^{[A,B]}  \big)    \big\rangle_{[A,B\cup\{j\}]}}  \\
        &\leq \frac14\frac{ \big\langle \exp \big(-B_j^{[A,B]}  \big)       \big\rangle_{[A,B\cup\{j\}]}^2 }{ \cosh^2\big( \big\langle   B_j^{[A,B]}      \big\rangle_{[A,B\cup\{j\}]}\big) } \frac{\big\langle \cosh ^2\big(B_j^{[A,B]}  \big)    \big\rangle_{[A,B\cup\{j\}]}^{1/2}}{\big\langle \cosh \big(B_j^{[A,B]}  \big)    \big\rangle_{[A,B\cup\{j\}]}} \big\langle  \exp \big(4 B_j^{[A,B]} \big)\big\rangle_{[A,B\cup\{j\}]}^{1/2}\\
        &\leq \frac14 (1+4\epsilon)^4  \big\langle \cosh^2 \big( B_j^{[A,B]}  \big)\big\rangle_{[A,B\cup\{j\}]}^{-1} 
        \end{split}
        \end{equation}
as well as
          \begin{equation} \label{eq:lm41aux3}
          \begin{split}
         & \frac14 \frac{ \big\langle \cosh \big(B_j^{[A,B]}  \big) + \sinh \big(B_j^{[A,B]}  \big)   \big\rangle_{[A,B\cup \{j\}]}^2 }{\big\langle \cosh \big(B_j^{[A,B]}  \big)    \big\rangle_{[A,B\cup\{j\}]}^2 }  \frac{ \big\langle  \cosh \big(B_j^{[A,B]}  \big)\exp \big(-2 B_j^{[A,B]} \big)\big\rangle_{[A,B\cup\{j\}]} }{\big\langle \cosh \big(B_j^{[A,B]}  \big)    \big\rangle_{[A,B\cup\{j\}]}}\\
        &\leq \frac14\frac{ \big\langle \exp \big(B_j^{[A,B]}  \big)       \big\rangle_{[A,B\cup\{j\}]}^2 }{ \cosh^2\big( \big\langle   B_j^{[A,B]}    \big\rangle_{[A,B\cup\{j\}]}\big) } \frac{\big\langle \cosh ^2\big(B_j^{[A,B]}  \big)    \big\rangle_{[A,B\cup\{j\}]}^{1/2}}{\big\langle \cosh \big(B_j^{[A,B]}  \big)    \big\rangle_{[A,B\cup\{j\}]}}   \big\langle  \exp \big(-2 B_j^{[A,B]} \big)\big\rangle_{[A,B\cup\{j\}]}\\
        &\leq  \frac14(1+4\epsilon)^4  \big\langle \cosh^2 \big( B_j^{[A,B]}  \big)\big\rangle_{[A,B\cup\{j\}]}^{-1} ,
        \end{split}
        \end{equation}
where we used repeatedly the bound \eqref{eq:bd-m-3} from Lemma \ref{bd-m}. Using also the first bound \eqref{eq:bd-m-1} from Lemma \ref{bd-m}, the previous bounds yield altogether that 
        \begin{equation*}\label{eq:lm41aux1}
            \frac{\big\langle\partial_j f;\sigma_j\big\rangle_{[A,B]}^2} {1-\big(m^{[A,B]}_j\big)^2}  \leq (1+4\epsilon)^4\big\langle \cosh^{-2}\big(B_j^{[A,B]} \big) (\partial_j f)^2 \big\rangle_{[A,B]}.
        \end{equation*}
Combining this bound with Lemma \ref{lem:Condition}, we obtain
    \begin{align*}
     \langle f ; f \rangle_{[A,B]} \le & a_{N-k-1} \Big( 1- \frac{1}{N-k} \Big)  D_{[A,B]}(f) + \frac{ (1+\epsilon ) (1 + 4\epsilon)^4}{N-k} D_{[A,B]}(f) \nonumber \\
     &+\frac{1}{2 \epsilon }\frac{1}{N-k}\sum_{j \not \in A \cup B} \frac{\langle f ;h^{[A,B]}_j \rangle_{[A,B]}^2}{1- (m^{[A,B]}_j)^2}. 
    \end{align*}
To conclude the lemma, the last term on the r.h.s. of the last equation can be estimated similarly as in \eqref{eq:f=csigma+h}.
\end{proof}

To obtain the improved iteration bound of Proposition \ref{thm:Dichotomy}, the previous lemma suggests to study the norm of $S^{[A,B]} $, introduced in \eqref{eq:def-H}. 
We do this in two main steps and, similarly as in the previous section, we need a technical preparation whose proof is explained in Section \ref{sec:est}. 

\begin{Lem} \label{lem:Xbound}
Let $\beta$ be as in \eqref{eq:defbeta}, $\Omega$ as in \eqref{eq:def-omg}, $ A,B\subset \{1,\dots,N\}$ be disjoint with $ |A\cup B| = k$ and for $q\in \mathbb{N}$, let
		\begin{equation*}
   		I_q^{[A,B]}: =\max \bigg[ \max_{i\not \in A\cup B} \sum_{j \not \in A\cup B}\big |\partial_i B_j^{[A,B]}\big|^q,\max_{j\not \in A\cup B} \sum_{i \not \in A\cup B} \big|\partial_i B_j^{[A,B]}\big|^q \bigg]\,. 
		\end{equation*} 
Then, for every $q\geq 2$, we have in $\Omega$ that  $ I_q^{[A,B]} \leq (5\bn)^q\,.$

Moreover, defining $ X_q^{[A,B]}$ as the matrix with entries
        \[ X_{q,ij}^{[A,B]}:= \big( \partial_i B^{[A,B]}_j\big)^{q}\hspace{0.5cm} \text{ for } 1\leq i,j\leq N-k,  \]
the matrix norm of $ X_q^{[A,B]}$ is bounded in $\Omega$ by $ \big\| X_q^{[A,B]} \big \|\leq (5\beta)^q$, for every $q\in\mathbb{N}$.
\end{Lem}
With the input of Lemma \ref{lem:Xbound}, we can prove the following result. 
\begin{lemma} \label{lem:MatrixBound}
Let $\beta$ be as in \eqref{eq:defbeta}, $\Omega$ as in \eqref{eq:def-omg}, $C_\beta$ as in \eqref{eq:def-cbeta}, $  \epsilon \in (0,10^{-2})$, $ A,B\subset \{1,\dots,N\}$ disjoint with $ |A\cup B| = k$ and let $S^{[A,B]}$ be as in \eqref{eq:def-H}. Assume that $\beta$ is sufficiently small and set
        \begin{equation*}
        \Omega_{ N-k-1,\epsilon}  = \big\{  a_{N -k-1} \,C_\beta < \epsilon \big\}.
        \end{equation*}
Then, there exists some universal $C>0$ such that in $\Omega \cap \Omega_{ N-k-1,\epsilon} $, we have
        \begin{equation*}
            \big\| S^{[A,B]} \big\| \le C \bn^2\exp\left( C\bn\right) \max(1, a_{[A,B]})^2  .
        \end{equation*}
\end{lemma}
\begin{proof}
We split the bound into two main steps, based on
        \begin{equation} \label{eq:applyspecgap-1}
        \begin{split}
       0&\leq  \sum_{i,j \not \in A \cup B} c_i S^{[A,B]}_{ij} c_j = \Ex{ \sum_{i \not \in A \cup B }c_i \frac{h_i^{[A,B]}}{\sqrt{1-(m^{[A,B]}_i)^2}}; \sum_{j \not \in A \cup B }c_j \frac{h_j^{[A,B]}}{\sqrt{1-(m^{[A,B]}_j)^2}} }_{[A,B]}\\
          & \le a_{[A,B]} \sum_{i \not \in A \cup B} \bigg\langle \cosh^{-2}(B_i^{[A,B]}) \bigg( \sum_{ j \not \in A \cup B}    \frac{\partial_i h_j^{[A,B]}. }{ \sqrt{1-(m^{[A,B]}_j)^2} }c_j\bigg)^2 \bigg\rangle_{[A,B]}\\
         &\le 2 a_{[A,B]} \sum_{i \not \in A \cup B} \bigg\langle \cosh^{-2}(B_i^{[A,B]})\bigg(\sum_{j \not \in A \cup B\cup\{i\}}   \frac{\partial_i h^{[A,B]}_j}{\sqrt{1-(m^{[A,B]}_j)^2}}c_j\bigg)^2 \bigg\rangle_{[A,B]}\\
         & \hspace{0.5cm}+ 2 a_{[A,B]} \sum_{i \not \in A \cup B} \frac{c_i^2}{{1-(m^{[A,B]}_i)^2}} \Big\langle \cosh^{-2}(B_i^{[A,B]}) (\partial_i h^{[A,B]}_i)^2 \Big\rangle_{[A,B]}\\
         &=: \text{T}_1 + \text{T}_2
        \end{split}\end{equation}
for $\textbf{c}\in \mathbb{R}^{N- k} $. We bound the contributions $ \text{T}_1$ and $\text{T}_2$, defined on the r.h.s. in \eqref{eq:applyspecgap-1}, separately. To this end, a straightforward computation yields first of all that 
        $$\partial_i h_j^{[A,B]}(\sigma) =
        \begin{cases}
        - \frac{1}{2}e^{-2\sigma_j B_j^{[A,B]}} ( m^{[A,B]}_j + \sigma_j) X^{[A,B]}_{ij},& i \not = j, \\
        -\sigma_j \sinh \big(2B_j^{[A,B]}\big) \big(\tanh\big(B_j^{[A,B]}\big)-m^{[A,B]}_j\big),&i  = j,
        \end{cases}$$
 where $X^{[A,B]} = \big( X^{[A,B]}_{ij}\big)_{1\leq i,j\leq N-k}$ is defined by
 	\begin{equation*} 
     	X^{[A,B]}_{ij}= 1 - \exp\big( 2\sigma_{j} \partial_i B_j^{[A,B]}\big) = \sum_{q=1}^\infty \frac{2^q}{q!} \big( \sigma_{j} \partial_i B_j^{[A,B]}\big)^q.
	\end{equation*}
 In particular, Lemma \ref{lem:Xbound} implies that in $\Omega$, we have that
        \[ \big\| X^{[A,B]}\big\| \leq \sum_{q=1}^\infty \frac{2^q}{q!}  \big\| X_q^{[A,B]}\big\| \leq C\beta \exp (C\beta) \]
for some universal $C>0$ and thus 
        \[\begin{split}
            \text{T}_1 & \leq   C\beta^2 \exp (C\beta)a_{[A,B]} \!\!\!\sum_{j \not \in A \cup B\cup\{i\}} \!\!\! \big( 1-\big(m^{[A,B]}_j\big)^2\big)^{-1}   \big\langle e^{-4\sigma_j B_j} ( m^{[A,B]}_j + \sigma_j)^2\big\rangle_{[A,B]} \, c_j^2.
        \end{split}\]
To control the r.h.s. further, we compute
        \[\begin{split}
            &  \big\langle e^{-4\sigma_j B_j^{[A,B]}} ( m^{[A,B]}_j + \sigma_j)^2\big\rangle_{[A,B]}  \\
            & =  \frac{ \big(1+ \big(m^{[A,B]}_j\big)^2\big)\big\langle \cosh\big( 3 B_j^{[A,B]}\big)  \big\rangle_{[A,B\cup \{j\}]}   - 2  m^{[A,B]}_j    \big\langle \sinh \big( 3 B_j^{[A,B]}\big)  \big\rangle_{[A,B\cup \{j\}]} }{ \big\langle \cosh\big(  B_j^{[A,B]}\big)  \big\rangle_{[A,B\cup \{j\}]} }\\
            & =\frac12 \frac{ \big(1 -  m^{[A,B]}_j \big)^2 \big\langle  e^{ 3 B_j^{[A,B]}}  \big\rangle_{[A,B\cup \{j\}]} + \big(1+  m^{[A,B]}_j )^2   \big\langle  e^{- 3 B_j^{[A,B]} }  \big\rangle_{[A,B\cup \{j\}]} }{ \big\langle \cosh\big(  B_j^{[A,B]}\big)  \big\rangle_{[A,B\cup \{j\}]} }. 
         \end{split}\]
Now, recalling once more that $m^{[A,B]}_j = \frac{ \langle \sinh (B_j^{[A,B]}  )    \rangle_{[A,B\cup\{j\}]}}{\langle \cosh (B_j^{[A,B]}  )    \rangle_{[A,B\cup\{j\}]}} $, a similar computation as in \eqref{eq:lm41aux2} and \eqref{eq:lm41aux3} using Lemma \ref{bd-m} shows that in $  \Omega \cap \big\{ a_{N-k-1,\epsilon} C_\beta < \epsilon\big\} $
        \[\big\langle e^{-4\sigma_j B_j^{[A,B]}} ( m^{[A,B]}_j + \sigma_j)^2\big\rangle_{[A,B]} \leq C   \big\langle \cosh^2 \big( B_j^{[A,B]}  \big)\big\rangle_{[A,B\cup\{j\}]}^{-1} \leq  C \big(1 -  \big(m^{[A,B]}_j \big)^2\big)  \]
for some constant $C>0$, that is independent of all parameters. Consequently, it holds true that $ \text{T}_1\leq C\beta^2 \exp (C\beta)a_{[A,B]} \| \textbf{c}\|_2^2$. 
 
Next, we consider the contribution $\text{T}_2$, defined in \eqref{eq:applyspecgap-1}. In this case, in order to extract a factor $\beta$, we need to apply the spectral gap inequality again.  Setting
		\[ u_i :=\tanh\big(B_i^{[A,B]}\big) -m^{[A,B]}_i, \]
we apply Cauchy-Schwarz and find
	\[\begin{split}
        &\Ex{\cosh^{-2}(B_i^{[A,B]})\big(\partial_i h_i^{[A,B]}\big)^2}_{[A,B]} \\
        &= \frac{ \big\langle 4\sinh^2\big(B_i^{[A,B]}\big)\cosh\big(B_i^{[A,B]}\big) (\tanh\big(B_i^{[A,B]}\big )-m^{[A,B]}_i\big)^2\big\rangle_{[A,B \cup\{i\} ]} } { \big\langle \cosh\big(B_i^{[A,B]}\big)  \big\rangle_{[A,B \cup\{i\}} }\\
        &\leq C \frac{ \big\langle \cosh^6\big( B_i^{[A,B]} \big) \big\rangle_{[A,B\cup\{i\} ]}^{1/2} }{\big\langle \cosh\big(B_i^{[A,B]}\big)  \big\rangle_{[A,B \cup\{i\}} }  \Ex{u_i^4}^{1/2}_{[A,B]}.
        \end{split}\]	
Then, since $ \langle u_i \rangle_{[A,B]} =0 $, we can bound $ \Ex{u_i^4}_{[A,B]} \leq \Ex{u_i^2;u_i^2}_{[A,B ]} + \Ex{u_i;u_i}_{[A,B]}^2$
and by the spectral gap inequality, we get
		\begin{equation*}
		\begin{split}
   		\Ex{u_i;u_i}_{[A,B]} & \leq a_{[A,B]} \sum_{k\not \in A\cup B}\Ex{   (\partial_k u_i)^2}_{[A,B]}, \\
		\Ex{u_i^2;u_i^2}_{[A,B]} & \leq 4 \, a_{[A,B]} \sum_{k\not \in A\cup B}\Ex{(\partial_k u_i )^2(u_i - \partial_k u_i )^2}_{[A,B]}\,.
		\end{split}
		\end{equation*}
Using the elementary bound $ |\tanh(x) - \tanh(y)|\leq 2  \cosh^{-2}(x) |x-y| e^{|x-y|}$ for $x,y\in \mathbb{R}$, we then obtain
		\[ \big| \partial_k u_i^{[A,B]} \big|\leq 2  \cosh^{-2}\big(B_i^{[A,B]} \big) \big|\partial_k B_i^{[A,B]} \big|  \exp\big( \big|\partial_k B_i^{[A,B]} \big| \big)\]
so that Lemma \ref{lem:Xbound} implies for $\beta$ small enough (so that $5\beta < 1$) that
		\[\begin{split}
		\Ex{u_i;u_i}_{[A,B]} & \leq 4 \,a_{[A,B]} \sum_{k\not \in A\cup B}\Ex{   \cosh^{-4}\big(B_i^{[A,B]} \big) \big|\partial_k B_i^{[A,B]} \big|^2  \exp\big( 2\big|\partial_k B_i^{[A,B]} \big| \big)}_{[A,B]}\\
		& \leq C \beta^2 \exp(C \beta )a_{[A,B]} \big\langle   \cosh^{-4}\big(B_i^{[A,B]} \big)\big\rangle_{[A,B]}\\
		& = C \beta^2 \exp(C \beta )a_{[A,B]}  \frac{ \big\langle   \cosh^{-3}\big(B_i^{[A,B]} \big)\big\rangle_{[A,B\cup \{i\} ]} }{ \big\langle   \cosh \big(B_i^{[A,B]} \big)\big\rangle_{[A,B\cup \{i\} ]}  }
		\end{split}\]
Analogously, we obtain that 
		\[\Ex{u_i^2;u_i^2}_{[A,B]}\leq  C \beta^2 \exp(C \beta )a_{[A,B]}  \frac{ \big\langle   \cosh^{-3}\big(B_i^{[A,B]} \big)\big\rangle_{[A,B\cup \{i\} ]} }{ \big\langle   \cosh \big(B_i^{[A,B]} \big)\big\rangle_{[A,B\cup \{i\} ]}  }\]
and combining this with the previous estimates, we conclude that 
		\[\begin{split}
		\text{T}_2 
		&\leq C a_{[A,B]}^{3/2} \sum_{i \not \in A \cup B} \frac{\beta^2 \exp(C \beta )c_i^2}{{1-(m^{[A,B]}_i)^2}} \frac{ \big\langle \cosh^6\big( B_i^{[A,B]} \big) \big\rangle_{[A,B\cup\{i\} ]}^{1/2} }{\big\langle \cosh\big(B_i^{[A,B]}\big)  \big\rangle_{[A,B \cup\{i\}]}^2 } \Ex{   \cosh^{-3}\big(B_i^{[A,B]} \big)}_{[A,B\cup\{i\}]} \\
		&\leq C\beta^2 \exp(C \beta ) \max(1,a_{[A,B]})^2  \| \textbf{c}\|_2^2
		\end{split}\]
Here, the second step follows from Lemma \ref{bd-m}, arguing as before. In conclusion, we have shown that for all $\textbf{c}\in \mathbb{R}^{N-k}$, we have that
		\[ 0\leq \langle \textbf{c},  S^{[A,B]}\textbf{c}\rangle_2 \leq C\beta^2  \exp(C \beta ) \max( 1, a_{[A,B]})^2   \| \textbf{c}\|_2. \]
for some universal $C>0$, i.e. $  \| S^{[A,B]}\| \leq C\beta^2  \exp(C \beta ) \max( 1, a_{[A,B]})^2 $. 
\end{proof}
  
\begin{proof}[Proof of Prop. \ref{thm:Dichotomy}]
We combine Lemmas \ref{lm:41} and \ref{lem:MatrixBound}, which implies directly the improved iteration bound in Proposition \ref{thm:Dichotomy}. 
\end{proof} 

\section{Proof of Theorem \ref{thm:main}} \label{sec:pfmain}
In this section, we combine Prop. \ref{continuity} and Prop. \ref{thm:Dichotomy} to prove Theorem \ref{thm:main}. 

\begin{proof}[Proof of Theorem \ref{thm:main}]

W.l.o.g., we can assume that $ \epsilon$ is sufficiently small. In particular, we assume in the following that $ \epsilon \in (0,10^{-2})$ so that we can apply the main results of the previous sections.

We proceed inductively and before giving the details, let us first of all show that $a_1=1$. To see this, let $A, B\subset \{1,\dots, N\}$ be disjoint such that $ |A\cup B| =N-1$ and let $ i \not \in A\cup B$. Then, the Gibbs measure $ \langle\cdot \rangle_{[A,B]}$ is simply the coin tossing measure associated to the external field $ B_i^{[A,B]} $. In other words, it is determined by the probabilities $ \langle \textbf{1}_{\sigma_i = \pm 1} \rangle_{[A,B]} = \frac12  \pm \frac12 \tanh\big( B_i^{[A,B]} \big) $. Set $ p_i:= \frac12 + \frac12 \tanh\big( B_i^{[A,B]} \big)$ and let $ f: \{\pm1\}\to \mathbb{R}$ be a function s.t. $ f(1)\neq f(-1)$. Moreover, assume without loss of generality that $ \langle f \rangle_{[A,B]} =0$, i.e. that 
		$$ f(-1) = - \frac{p_i}{1-p_i} f(1).$$ 
Then, we find that
		\[\begin{split}
		\langle f^2 \rangle_{[A,B]} = p_i f(1)^2 + (1-p_i)f(-1)^2 &=  \frac{p_i}{1-p_i} f(1)^2 \\
		&= p_i (1-p_i) \big\langle  \big( f(1)- f(-1)\big)^2 \big\rangle_{[A,B]} =D(f) 
		\end{split}\]
and thus $ a_{[A,B]} = 1$. Since $A, B\subset \{1,\dots, N\}$ were arbitrary, this means that $a_1=1$. 

Now, let $\epsilon >0 $ be sufficiently small. We choose the inverse temperature $ \beta =  \epsilon^{3/4} \ll 1 $ so that  
		\[ \frac{\beta}{\epsilon} = \epsilon^{-1/4} \gg 1  \hspace{0.5cm} \text{ and } \hspace{0.5cm}  \frac{\beta}{\epsilon^{1/2} } = \epsilon^{1/4} \ll 1.   \]
In particular, we can assume in the following w.l.o.g. that $\epsilon$ is so small that 
		\[\begin{split}
		(i) &\hspace{0.5cm} \frac{ C_\beta^{2} }{C_1 \epsilon^{2}  \beta^2 } = \mathcal{O}(1)\epsilon^{-1/2} > 2 \log 4, \\
		(ii) &\hspace{0.5cm} (1+  40\epsilon^{1/4} ) \,C_\beta  = \mathcal{O}(1)\, \epsilon^{3/2}   < \epsilon  \hspace{0.5cm} \text{ and } \\
		(iii) &\hspace{0.5cm} 25 (1+40\epsilon^{1/4} )^2  \frac{ C_2 \beta^2 e^{C_2 \beta } }{\epsilon }  = \mathcal{O}(1) \epsilon^{1/2} < \epsilon^{1/4}, 
		\end{split}\]
where each $\mathcal{O}(1)>0$ is bounded by some universal constant, and the constants $C_1 $ and $C_2$ refer to the universal constants from Prop. \ref{continuity} and Prop. \ref{thm:Dichotomy}, respectively.

Next, the event $ \widetilde \Omega$ of high probability we use for the proof of Theorem \ref{thm:main} is defined by 
        \[\begin{split} \widetilde \Omega:&= \Omega\cap  \bigcap_{l=0}^{N-2} \Big\{  a_{N-l-1} C_\beta < \epsilon \text{ and } a_{N-l}   >   5a_{N-l-1} \Big\}^c\\
        &= \Omega\cap  \bigcap_{l=0}^{N-2} \Big\{  a_{N-l-1} C_\beta \geq  \epsilon \text{ or } a_{N-l}   \leq   5a_{N-l-1} \Big\},
        \end{split}\]
where $\Omega$ denotes the set defined in \eqref{eq:def-omg}. In particular, according to Prop. \ref{continuity} combined with condition $(i)$ from above as well as the fact that $\mathbb{P}(\Omega) \geq 1- e^{-N}$, it holds true that  
        \[ \mathbb{P} (\widetilde\Omega  )   \geq  1 - Ce^{-cN} \]
for universal constants $c, C>0$. 

Now, let us prove the following induction: if $a_{N-k-1}\leq 1 +  40\epsilon^{1/4}$ in $\widetilde \Omega$ for some $0\leq k \leq N-2$, then we prove that this implies $ a_{N-k} \leq 1 +  40 \epsilon^{1/4}$ in $\widetilde \Omega$. Since $a_1 = 1$, the inductive assumption is clearly satisfied for $k=N-2$. Hence, Theorem \ref{thm:main} follows from the proof of the inductive step.

To prove the inductive step, by the inductive assumption on $a_{N-k-1}$ and condition $(ii)$ from above, we have that $a_{N-k-1} C_\beta < \epsilon $ in $\widetilde\Omega$. By definition of $\widetilde\Omega$, this implies that
        $$ a_{N-k} \leq 5 a_{N-k-1} < 5\,(1+40\epsilon^{1/4})  $$ 
in $\widetilde\Omega$. Therefore, Proposition \ref{thm:Dichotomy} implies that
		\[\begin{split}
		\bigg(1 -  \frac{ 25 C_2\beta^2 e^{C_2\beta}\,  (1+40\epsilon^{1/4})^2 }{\epsilon  \, (N-k)} \bigg)  a_{N-k} & \leq \Big(1- \frac{1}{N-k}\Big)a_{N-k-1} + \frac{ (1+4\epsilon )^5}{N-k}\\
		&\leq \Big(1- \frac{1}{N-k}\Big)a_{N-k-1} + \frac{1 +  32 \epsilon^{1/4} }{N-k}.
		\end{split} \]
Combining this with condition $(iii)$ from above and using again $ a_{N-k} \leq 5 a_{N-k-1} $, we obtain that
		\[\begin{split}
		a_{N-k} & \leq \Big(1+ \frac{5 \epsilon^{1/4}}{N-k}- \frac{1}{N-k}\Big)a_{N-k-1} + \frac{1 +  32 \epsilon^{1/4} }{N-k}  \\
		&\leq \Big(1 + \frac{5 \epsilon^{1/4}}{N-k}- \frac{1}{N-k}\Big) ( 1+ 40\epsilon^{1/4}) + \frac{1 +  32 \epsilon^{1/4} }{N-k}   \\
		& = 1 + 40\epsilon^{1/4} +  \frac{ (5\epsilon^{1/4})(40\epsilon^{1/4})  - 3\epsilon^{1/4}}{N-k}  \leq 1 + 40\epsilon^{1/4},
		\end{split}\]
for all $\epsilon >0 $ sufficiently small. This proves the inductive step and therefore, arriving at $ k=0$, we have shown that on a set of probability at least $1- C e^{-cN}$, we have that $  a_{H_N}   \leq 1 + 40\epsilon^{1/4} $.
\end{proof}

\section{High probability Estimates}
\label{sec:est}

In this section, we complete our arguments by providing the high probability estimates related to the set $\Omega$, defined in \eqref{eq:def-omg}, Lemma \ref{bd-m} as well as Lemma \ref{lem:Xbound}.

\begin{Lem} \label{lem:partialbnd}
Let $\Omega$ be defined as in \eqref{eq:def-omg}. Then, $\mathbb{P}(\Omega) \geq 1 - e^{-N}$.
\end{Lem}
\begin{proof}
We recall from \eqref{eq:BjAB} that
		\[\begin{split}
		B_j^{[A,B]}(\sigma) = &\;\sum_{p\geq2} \frac{\beta_p}{N^{(p-1)/2}}\bigg( \sum_{  i_2,\ldots,i_p \in B^c,   } g_{j i_2\ldots i_p} \sigma_{i_2} \ldots \sigma_{i_p} +\ldots  \\
                &\hspace{2.5cm}\ldots+ \!\!\!\! \sum_{  i_1,\ldots,i_{p-1} \in B^c} g_{i_1 i_2\ldots i_{p-1}j} \sigma_{i_1}\sigma_{i_2} \ldots \sigma_{i_{p-1}}  \bigg) +   \eta_j 
                \end{split}\]
for $\sigma \in \Sigma_{N-|A\cup B|}$ and $\sigma_A$ understood to be fixed. It is then simple to see that an explicit expression for $\sigma_i \partial_i B_j^{[A,B]}$ is given by
		\begin{equation*} 
        		\sigma_i \partial_i B_j^{[A,B]} (\sigma) = \sum_{p\geq 2} \frac{\beta_p}{\sqrt{N}} \sum_{k_1, k_2, \ldots, k_{p-2} \in B^c} g_{k_1k_2\ldots k_{p-2}}^{ij} \frac{\sigma_{i_1}}{\sqrt{N}}\ldots \frac{\sigma_{i_2}}{\sqrt{N}}\ldots \frac{\sigma_{i_{p-2}}}	
		{\sqrt{N}},
    		\end{equation*}
where for each $  i,j \in \{ 1,\dots, N-|A\cup B| \}$, the coupling $g_{k_1k_2\ldots k_{p-2}}^{ij}$ is a sum of $p(p-1)$ i.i.d. standard Gaussian random variables (if the indices $k_u\neq k_s$ for $1\leq u\neq s\leq p-2$ are all distinct; otherwise $g_{k_1k_2\ldots k_{p-2}}^{ij} =0$). Moreover, the couplings $g_{k_1k_2\ldots k_{p-2}}^{ij}$ are independent over the superscripts $  i,j \in \{ 1,\dots, N-|A\cup B| \}$. We can view this as 
		\[\begin{split}
		\sigma_i \partial_i B_j^{[A,B]} (\sigma)=  \sum_{p=2}^\infty \frac{\beta_p}{\sqrt{N}} G_p (\tau , \tau,  \ldots , \tau , \mathbf e_i , \mathbf e_j), 
		\end{split}\]
 where $G_p:(\mathbb{R}^N)^p\to \mathbb{R}$ denotes the multilinear map that is defined through \linebreak $(G_p)_{k_1k_2\dots k_p}:=  g_{k_1,\ldots,k_{p-2}}^{k_{p-1} k_p}$, the vector $ \tau \in \{ \textbf{y}\in \mathbb{R}^N: \|\textbf{y}\|_2\leq 1\}$ has components $ \tau_i = 0$ if $i\in B$ and $\tau_i = N^{-1/2} \sigma_i$ if $i\in \{1,\dots,N\}\setminus B$ and $(\textbf{e}_k)_{k=1,\dots, N}$ denotes the standard basis of $\mathbb{R}^N$. In particular, we have that
 		\[\begin{split}
		\sup_{ \substack{ A,B\subset \{1,\dots, N\}: \\A\cap B=\emptyset  }}\sup_{\sigma_A\in \Sigma_{|A|} } \sup_{\sigma \in \Sigma_{N-|A\cup B|} } \| \partial_i B_j^{[A,B]} (\sigma)\| \leq \sum_{p=2}^\infty  \frac{\beta_p}{\sqrt{N}}  |\!|\!| G_p  |\!|\!|, 
		\end{split}\]
 where 
    		\[  |\!|\!| G_p  |\!|\!| := \sup_{ \substack{  \textbf{y}_1,\ldots,\textbf{y}_p \in \mathbb{R}^N: \\ \| \textbf{y}_i\|_2 \leq  1 \,\forall i=1,\dots, p  }} G_p(\textbf{y}_1\otimes \textbf{y}_2...\otimes \textbf{y}_p). \]
According to \cite[Theorem 1]{TS}, we have that 
    $$ \P\Big( \Big\{  N^{-1/2}   |\!|\!| G_p  |\!|\!|  >  \sqrt{p(p-1)[8 \log(p/\log(3/2)) p + \log t+\log 2 ]} \Big\} \Big) \leq e^{ - N \log t }$$
for any $t > 1$. Choosing $t = 2p $ for $p\geq 2$, this implies 
		\[ \P\big( \big\{  N^{-1/2}   |\!|\!| G_p  |\!|\!|  >  5 \sqrt{ p^3 \log p }  \big\} \big) \leq e^{ - N \log 2p }. \] 
Finally, since $   \sup_{ \substack{ A,B\subset \{1,\dots, N\}: \\A\cap B=\emptyset  }}\sup_{\sigma_A\in \Sigma_{|A|} } \sup_{\sigma \in \Sigma_{N-|A\cup B|} } \| \partial_i B_j^{[A,B]} (\sigma)\|    > 5\beta $ for $\beta = \sum_{p=2}^\infty  \sqrt{p^3 \log p } \beta_p	$, defined in \eqref{eq:defbeta}, implies that $N^{-1/2}   |\!|\!| G_{p_*}  |\!|\!|  >  5 \sqrt{ p_*^3 \log p_* }$ for some $p_*\geq 2$, we conclude that $\mathbb{P}(\Omega) \geq 1- e^{-N \log 4} \sum_{p=2}\big(\frac{2}{p}\big)^N \geq 1- e^{-N} $. 		
\end{proof}

We continue with the proof of Lemma \ref{lem:Xbound}. 
\begin{proof}[Proof of Lemma \ref{lem:Xbound}]
Assuming the first claim $ I_q^{[A,B]} \leq (5\beta )^q$ for $q\geq 2$ for the moment, the bounds on the matrices $X_q^{[A,B]}$ are simple. Indeed, for $q=1$, we have $X_{1,ij}^{[A,B]} = \partial_i B_j^{[A,B]}$, so $ \| X_1^{[A,B]}\| \leq 5\beta  $ in $\Omega$, by definition of $\Omega$. For $q\geq 2$, on the other hand, we have that $ \| X_q^{[A,B]}\|\leq I_q^{[A,B]}$ by standard matrix properties and hence the claim follows if $ I_q^{[A,B]} \leq (5\beta )^q$ for $q\geq 2$. To this end, notice for $q=2$
		\[\begin{split}
		I_2^{[A,B]}  &= \max_{1 \leq i \leq N-|A\cup B|} \max \Big (  \big[\big(X_{1}^{[A,B]}\big)^t\big(X_{1}^{[A,B]}\big) \big]_{ii},  \big[\big(X_{1}^{[A,B]}\big)\big(X_{1}^{[A,B]}\big)^t \big]_{ii}\Big) \\
		& \leq  \| X_1^{[A,B]}\|^2 \leq (5\beta)^2
		\end{split}\]
and for $q  > 2$ that $I_q \leq I_{2}  \max_{ 1\leq i,j \leq N-|A\cup B| } |\partial_i B_j^{[A,B]} |^{q-2}\leq (5\beta)^q$ in $\Omega$. 
\end{proof}
Finally, it remains to prove Lemma \ref{bd-m} which follows as a simple corollary from the next Lemma. 
\begin{Lem} \label{lem:B-Bbar}
Let $\epsilon \in (0,10^{-2})$ and $K \in \mathbb R$. Suppose that $A, B \subset \{1,\dots, N\}$ are disjoint and let $j \not \in A\cup B$. Then, if $\beta$ is sufficiently small (depending on $K$) and if
    \begin{equation*}
     C_{K,\beta}a_{[A,B \cup \{j\}]} :=     (20K\bn)^{2}\exp(20K \bn)a_{[A,B \cup \{j\}]}< \epsilon\,,
    \end{equation*}
we have in $\Omega$ that
    \begin{equation*}
        \big \langle \exp(KB_j^{[A,B]}) \big\rangle_{[A,B]} \leq (1 + 4 \epsilon)\exp\big(K \langle B_j^{[A,B]}\rangle_{[A,B\cup \{j\}]}\big).
    \end{equation*}
\end{Lem}
\begin{proof}
Writing $ \langle f^2 \rangle_{[A, B\cup \{j\} ]} =   \langle f ;f \rangle_{[A, B\cup \{j\} ]} + \langle f \rangle^2_{[A, B\cup \{j\} ]}$, we bound
\begin{equation*}
    \langle f^2 \rangle_{[A, B\cup \{j\} ]} \le a_{[A, B\cup \{j\} ]} \sum_{k \not \in A \cup B\cup\{j\}}  \langle (\partial_ k f)^2  \rangle_{[A, B\cup \{j\} ]} + \langle f \rangle^2_{[A, B\cup \{j\} ]},
\end{equation*}
using the spectral gap inequality. Choosing $f= \exp\big(K B_j^{[A,B]}\big)$, we have \linebreak $\partial_k f =   \exp\big(K B_j^{[A,B]}\big)\big(1 - \exp\big(2 K \partial_k B_j^{[A,B]} \big)\big)$ and Lemma \ref{lem:Xbound} implies
		\begin{align*}
    		\sum_{k \not \in A \cup B\cup \{j\}}\big(1 - \exp\big(2 K \partial_k B_j^{[A,B]} \big)\big)^2 &\leq   \sum_{q \geq 2} \frac{(4K)^q}{(q-2)!}  \sum_{ k \not \in A \cup B} |\partial_k B_j|^q \\
		&\leq  \sum_{q \geq 2} \frac{(4K)^q} {(q-2)!} I_q^{[A,B]} \leq (20 K\bn)^2\exp(20 K\bn)
\end{align*}
Here, we used the elementary bound $(e^x - 1)^2 \leq |x|^2 e^{2|x|}$ for $x\in \mathbb{R}$. We conclude 
		\[ \sum_{k \not \in A \cup B\cup\{j\}} \langle (\partial_k f)^2\rangle_{[A, B\cup \{j\} ]} \le C_{K,\beta}  \langle f^2 \rangle_{[A, B\cup \{j\} ]}, \]
for $f= \exp\big(K B_j^{[A,B]}\big)$ and since $C_{K,\beta} a_{[A, B\cup \{j\} ]} < \epsilon$, this implies
  	\[ \langle f^2 \rangle_{[A, B\cup \{j\} ]} \le\big(1 - C_{K,\beta} a_{[A, B\cup \{j\} ]}\big)^{-1} \langle f \rangle^2_{[A, B\cup \{j\} ]}. \]
Now, $  f^{1/2}   = \exp\big(\frac{K}2 B_j^{[A,B]}\big)$ so we can iterate the previous step and find that
		\[\begin{split}
		&\big \langle\exp\big(K B_j^{[A,B]}\big) \big\rangle_{[A, B\cup \{j\} ]}\\
		&\le \prod_{l=0}^{\infty} \frac1{ \big( 1- C_{ 2^{-l}K ,\beta}  a_{[A, B\cup \{j\} ]} \big)^{2^l}  } \lim_{l \to \infty} \big\langle \exp\big(2^{-l}K B^{[A,B]}_j\big) \big\rangle_{[A, B\cup \{j\} ]}^{2^l}. 
		\end{split}\]
Noting that $\lim_{l \to \infty} \big\langle \exp\big(2^{-l}K B^{[A,B]}_j\big) \big\rangle_{[A, B\cup \{j\} ]}^{2^l} = \exp\big(K \langle B_j^{[A,B]} \rangle_{[A, B\cup \{j\} ]}\big) $ and  
		\[\begin{split}  
		&\log  \prod_{l=0}^{\infty} \frac1{ \big( 1- C_{ 2^{-l}K ,\beta}  a_{[A, B\cup \{j\} ]} \big)^{2^l}  } \\
		& \leq \sum_{l=0}^\infty  2^l (- \log) \Big( 1- 2^{-2l} \exp\big(20 K \beta (2^{-l} -1) \big) C_{ K ,\beta}  a_{[A, B\cup \{j\} ]} \Big)\leq 3\epsilon
		\end{split}\]
for $\beta$ small enough (depending on $K$, so that $ \exp (|20 K \beta|)$ is sufficiently close to one). All in all, we conclude that 
		$$ \big \langle\exp\big(K B_j^{[A,B]}\big) \big\rangle_{[A, B\cup \{j\} ]}\leq (1 + 4 \epsilon)\exp\big(K \langle B_j^{[A,B]}\rangle_{[A,B\cup \{j\}]}\big).$$ 
\end{proof}

\begin{proof}[\bf Proof of Lemma \ref{bd-m}]
This is a simple application of the previous lemma, noting that the definition of $C_\beta$ in \eqref{eq:def-cbeta} and the assumption $ C_\beta a_{[A,B\cup \{j\}]} <\epsilon $ allows us to apply Lemma \ref{lem:B-Bbar} uniformly in $K\in [-20,20]$, if $\beta$ is small enough (so that $ \exp( 20^2 \beta)$ is sufficiently close to one). Indeed, with this observation, we obtain directly \eqref{eq:bd-m-3} and \eqref{eq:bd-m-2}. The bound \eqref{eq:bd-m-1} follows similarly, realizing first that by Jensen's inequality and the identity $m_i^{[A,B]} = \big\langle\tanh(B_i^{[A,B]})\big\rangle_{[A,B]}$, we have that
	\[\begin{split}
 	 \frac{1}{1 -  (m^{[A,B]}_j)^2}  \leq \Ex{\frac{1}{1 -\tanh^2(B_j^{[A,B]})}}_{[A,B]}&=  \big \langle \cosh^2(B_j^{[A,B]}) \big\rangle_{[A,B]}  \\
	 &= \frac{\big \langle \cosh^3(B_j^{[A,B]}) \big\rangle_{[A,B\cup \{j\} ]}} {\big \langle \cosh(B_j^{[A,B]}) \big\rangle_{[A,B\cup\{j\}]}} 
	\end{split}\]  
and then applying twice \eqref{eq:bd-m-2}.   
\end{proof}


\vspace{0.2cm}
\noindent\textbf{Acknowledgements.} A. A. was sponsored by a Harvard Term-Time Graduate Fellowship as well as NSF grant DMS-2102842. C. B. acknowledges support by the Deutsche Forschungsgemeinschaft (DFG, German Research Foundation) under Germany’s Excellence Strategy – GZ 2047/1, Projekt-ID 390685813. The work of C. X. is partially funded by a Simons Investigator award. The work of H.-T. Y. is partially supported by the NSF grant DMS-1855509 and DMS-2153335 and a Simons Investigator award. 
\vspace{0.2cm}

\bibliography{Spectral}

\bibliographystyle{abbrv}

\end{document}